\documentclass[a4paper]{amsart}
\usepackage[utf8]{inputenc}
\usepackage[english]{babel}
\usepackage{fancyhdr}
\usepackage{amsmath}
\usepackage{amsfonts,amstext,url,geometry}
\usepackage{amssymb,amsthm,amscd,amsxtra,wasysym,graphicx}

\usepackage{xcolor}

\usepackage{hyperref}
\hypersetup{
	unicode=false,          
	pdftoolbar=true,        
	pdfmenubar=true,        
	pdffitwindow=false,     
	pdfstartview={FitH},    
	pdftitle={Chern-Ricci flows},    
	pdfauthor={Quang-Tuan Dang},     
	colorlinks=true,       
	linkcolor=blue,          
	citecolor=blue,        
	filecolor=black,      
	urlcolor=blue}           

\usepackage{cleveref}
\usepackage{indentfirst}
\newtheorem{theorem}{Theorem}
\newtheorem{lemma}[theorem]{Lemma}
\newtheorem{corollary}[theorem]{Corollary}

\newtheorem{proposition}[theorem]{Proposition}
\theoremstyle{definition}
\newtheorem{definition}[theorem]{Definition}

\newtheorem{remark}[theorem]{Remark}
\theoremstyle{plain}
\newtheorem{bigthm}{Theorem}

\numberwithin{theorem}{section}
\numberwithin{equation}{section}
 
\DeclareMathOperator \PSH {{\rm PSH}}
\DeclareMathOperator \loc {{\rm loc}}

\DeclareMathOperator \Vol {{\rm Vol}}

\DeclareMathOperator \tr {\textrm{tr}}

\def\om{\omega}
\def\tom{\Tilde{\omega}}

\def\f{\varphi}
\def\dc{dd^c}
\def\p{\psi}

\begin{document}
	\title{Pluripotential Chern-Ricci Flows}
	\author{Quang-Tuan Dang}
	\address{Laboratoire de Math\'ematiques d'Orsay, Universit\'e Paris-Saclay, CNRS, 91405, Orsay, France.}
	\curraddr{Institut de Mathematiques de Toulouse,  Universit\'e de Toulouse; CNRS,
118 route de Narbonne, 31400 Toulouse, France}
		\email{quang-tuan.dang@universite-paris-saclay.fr, quang-tuan.dang@math.univ-toulouse.fr}
		\urladdr{https://www.math.univ-toulouse.fr/~qdang/}
	\date{\today}
	\thanks{This work is partially supported by the ANR project PARAPLUI}
	\maketitle
	\begin{abstract} Extending a recent theory developed on compact K\"ahler manifolds by Guedj-Lu-Zeriahi \cite{guedj2018pluripotential,guedj2020pluripotential} and the  author~\cite{dang2021pluripotential},  
		we  define and study  pluripotential  solutions to degenerate parabolic complex Monge-Amp\`ere equations on compact Hermitian manifolds. Under natural assumptions on the Cauchy boundary data, we show that the pluripotential solution is semi-concave in time and continuous in space and that such a solution is unique. 
		We also establish a partial regularity of such solutions under some extra assumptions of the densities and apply it to prove the existence and uniqueness of the weak Chern-Ricci flow on complex compact varieties with log terminal singularities.  
	\end{abstract}	
	
	\tableofcontents
	
	\section*{Introduction}
	
	The Chern-Ricci flow is an evolution equation of Hermitian metrics  by their Chern-Ricci forms. It was first investigated by M. Gill \cite{gill2011convergence} in the setting of complex manifolds with vanishing first Bott-Chern class.  In \cite{tosatti2013chern, tosatti2015evolution, tosatti2015collapsing,zheng2017chern} the Chern-Ricci flow was studied  on more general complex manifolds and 
	 a number of further results were established, several of which are analogous to those for the K\"ahler-Ricci flow. These results provide affirmative evidence that the
Chern–Ricci flow is a natural geometric flow on complex manifolds and that its behavior
reflects the underlying geometry. Alternative  flows of Hermitian metrics  have been previously studied by Streets-Tian~\cite{streets2010parabolic,streets2011hermitian,streets2013regularity}, and also Liu-Yang~\cite{liu2012geometry}, motivated in part by  open classification problems. Following Song-Tian's  program~\cite{song2017kahler},  
our aim in this paper is to establish the existence of  the (weak) Chern-Ricci flow on mildly singular varieties.

	Let $(X,\omega_0)$ be a compact $n$-dimensional Hermitian manifold. The {\it Chern-Ricci flow} $\omega=\omega(t)$ starting at $\omega_0$ is an evolution equation of metrics \begin{align}\label{crf}
	    \frac{\partial}{\partial t}\omega=-{\rm Ric}(\omega), \quad\omega|_{t=0}=\omega_0,
	\end{align}
	where ${\rm Ric}(\omega)$ is the {\em Chern-Ricci form} of $\omega$ associated to the Hermitian metric $g=(g_{i\bar{j}})$, which in local coordinates is given by \[{\rm Ric}(\omega)=-\dc\log\det(g) .\]  In the K\"ahler setting, ${\rm Ric}(\omega)=iR_{j\bar{k}}dz_j\wedge d\bar{z}_k$, where $R_{j\bar{k}}$ is the Chern Ricci curvature of $\omega $. Thus
	if $\omega_0$ is K\"ahler i.e., $d\omega_0=0$, \eqref{crf} coincides with the K\"ahler-Ricci flow. 
		Tosatti and Weinkove \cite[Theorem 1.3]{tosatti2015evolution} showed that there exists a unique maximal solution  to \eqref{crf} on $[0,T)$ for a number $T\in (0,\infty]$ determined by $\omega_0$. 

	Solving the Chern-Ricci flow boils down to solving a parabolic scalar equation modeled on \begin{equation*}
	   \frac{\partial \f}{\partial t}=\log\frac{\det(g_{ij}+\partial_i\partial_{\bar{j}} \f)}{\det(g_{i\bar{j}})},\quad g_{ij}+\partial_i\partial_{\bar{j}} \f>0
	\end{equation*} with initial data $\f(0,x)=0$.
	\subsection*{Parabolic pluripotential method}
    Recently, Guedj-Lu-Zeriahi \cite{guedj2018pluripotential,guedj2020pluripotential} have developed the first steps of a parabolic pluripotential approach both in the local and K\"ahler cases.  This  enabled then to study the behavior of the Kähler–Ricci flow on compact K\"ahler varieties with Kawamata log terminal (klt) singularities, extending works by Cao \cite{cao1985deformation}, Tsuji \cite{tsuji1988existence}, Tian-Zhang \cite{tian2006kahler}, Song-Tian \cite{song2017kahler}, and Eyssidieux-Guedj-Zeriahi~\cite{eyssidieux2016weak,eyssidieux2018convergence}.
	
We remark that any compact complex manifold admits a Hermitian metric but there are many of them which are not K\"ahler.
It is thus desirable  to extend the previous works  \cite{guedj2018pluripotential, guedj2020pluripotential,dang2021pluripotential}
 to the Hermitian setting. It is expected that Hermitian analogues of the K\"ahler-Ricci flow will play an important role in understanding the geometry of compact complex manifolds.

More precisely,  we consider  the following  parabolic complex Monge–Ampère type equation
	\begin{align*}\label{cmaf} \tag{CMAF}
	dt\wedge(\omega_t+\dc\f_t)^n=e^{\partial_t{\f}_t+F(t,x,\f_t)}f(x)dV(x)\wedge dt
	\end{align*}
	in $X_T:=(0,T)\times X$, where $dV$ is a fixed normalized volume form on $X$ and
	\begin{itemize}
	\item $T\in (0,\infty)$, $X$ denotes a compact Hermitian manifold;
		\item $(\omega_t)_{t\in [0,T]}$ is a smooth family of semi-positive forms such that 
		\begin{equation}\label{assum: omega}
	    -A\omega_t\leq \dot{\omega}_t\leq A\omega_t\;\, \text{and}\;\, \ddot{\omega}_t\leq A\omega_t,\; \forall \, t\in [0,T],
	\end{equation} for some fixed constant $A>0$;
	\item for all $t\in [0,T]$, $\theta\leq \omega_t$ for some semi-positive and big $(1,1)$-form $\theta$, i.e. there exists a quasi-plurisubharmonic function $\rho$ with analytic singularities such that $\theta+\dc\rho$ dominates a Hermitian form;
		
		\item $(t,x,r)\mapsto F(t,x,r)$ is continuous on $[0,T]\times X\times \mathbb{R}$, quasi-increasing in $r$, uniformly Lipschitz in $(t,r)$, and uniformly convex in $(t,r)$;
		
		\item $0\leq f\in L^p(X)$ for some  $p>1$, and $f>0$ almost everywhere;
		
		\item $\f:[0,T]\times X\rightarrow \mathbb{R}$ is the unknown function, with $\f_t:=\f(t,\cdot)$. 
	\end{itemize}
	
	Here $d=\partial+\Bar{\partial}$ and $d^c=i(\Bar{\partial}-\partial)/2$ are both real operators, so that $\dc=i\partial\Bar{\partial}$. 
\medskip

Our plan is to extend the pluripotential approach developed by Guedj-Lu-Zeriahi \cite{guedj2020pluripotential} to the Hermitian context. We are going to introduce a notion of pluripotential solutions to~\eqref{cmaf}, a parabolic analogue of the theory developed  by Bedford and Taylor~\cite{bedford1976dirichlet,bedford1982new}.	 
	 The local side of this theory has been developed in \cite{guedj2018pluripotential} by a direct approach, taking advantage of the euclidean structure of $\mathbb{C}^n$. 
	 \medskip 
	 
	The above parabolic equation can be interpreted as a second-order PDE on the ($2n+1$)-dimensional manifold $X_T$:
	\begin{itemize}
	    \item the left-hand side $dt\wedge(\omega_t+\dc\f_t)$ is a well-defined positive Radon measure for the path $t\mapsto\f_t$ of bounded $\omega_t$-plurisubharmonic functions;
	    \item the right-hand side $e^{\dot{\f}_t+F(t,x,\f)}f(x)dV(x)\wedge dt$ is a well-defined positive Radon measure if $t\mapsto\f_t(x)$ is (locally) uniformly Lipschitz.  
 	\end{itemize}
	It is also useful in practice to allow the Lipschitz constant to blow up as $t$ approaches zero, so we introduce the corresponding class $\mathcal{P}(X_T,\omega)$ of \emph{parabolic potentials} (see Definition~\ref{def: pp})
	\medskip
	
	 We approximate the equation~\eqref{cmaf} by smooth parabolic complex Monge-Amp\`ere ones and establish various a priori estimates to prove our first main result:
	\begin{bigthm}
	\label{thmA} { Let $\f_0$ be a bounded $\omega_0$-psh function. Then there exists a  parabolic potential $\f\in\mathcal{P}(X_T,\omega)$ to  \eqref{cmaf} such that
\begin{itemize}
    \item $(t,x)\mapsto\f(t,x)$ is locally bounded in $[0,T)\times X$,
    \item $(t,x)\mapsto\f(t,x)$ is continuous in $(0,T)\times (X\setminus\{\rho=-\infty\})$,
    \item $t\mapsto\f_t$ is locally uniformly semi-concave in $(0,T)\times X$,
    \item $\f_t\to\f_0$ as $t\to 0^+$ in $L^1(X)$ and pointwise.
\end{itemize}}
	\end{bigthm}

 It turns out that  $t\to \f_t(x)-n(t\log t-t)+Ct$ is increasing for some fixed $C>0$. The convergence at time zero is therefore rather strong. For instance, it holds in the sense of capacity, and is even uniform if $\f_0$ is continuous. 
	
	
The semi-concavity property of the solution $\f$ constructed in Theorem~\ref{thmA} is a key ingredient for the approximation process (see \cite[Theorem 1.14]{guedj2020pluripotential}). We prove that it is the unique pluripotential solution with such regularity by establishing the following comparison principle:
	
\begin{bigthm}\label{thmB} { Let $\f\in\mathcal{P}(X_T,\omega)$ (resp. $\psi$) be a bounded pluripotential subsolution (resp. supersolution) to \eqref{cmaf} with initial data $\f_0$ (resp. $\psi_0$). Assume that  $\f$ is continuous in $(0,T)\times (X\setminus \{\rho=-\infty\})$ and $\psi$ is locally uniformly semi-concave in $t$. Then 
\begin{equation*}
    \f_0\leq \psi_0\Longrightarrow\f\leq \psi.
\end{equation*}
In particular, there is a unique pluripotential solution to \eqref{cmaf} which is continuous in $(0,T)\times (X\setminus \{\rho=-\infty\})$ and locally uniformly semi-concave in $t$.}
\end{bigthm}

\smallskip
We let $\Phi(\omega_t,F,f,\f_0)$ denote the unique solution to \eqref{cmaf} with given data $(\omega_t,F,f,\f_0)$ as in Theorem~\ref{thmB}. This comparison principle also allows us to  establish the following stability result:
\begin{bigthm}
\label{thmC}{Assume that $(\omega_t,F,f,\f_{0})$ and $(\omega_{t,j},F_j,f_j,\f_{0,j})$ satisfy the assumptions above with uniform constants independently of $j$, and
\begin{itemize}
\item $\omega_{t,j}$ are smooth Hermitian forms converging uniformly to $\omega_t$,
    \item $F_j$ uniformly converge  to $F$ with uniform constants,
    \item $f_j$ are densities which converge in $L^p(X)$ to $f$,
    \item $(\f_{0,j})$ is a sequence of  (smooth) bounded $\omega_0$-psh functions converging in $L^1(X)$ towards $\f_0\in \PSH(X,\omega_0)\cap L^\infty(X)$.
\end{itemize}
Then $\Phi(\omega_{t,j},F_j,f_j,\f_{0,j})$ locally uniformly converge  to $\Phi(\omega_{t},F,f,\f_{0})$.}
\end{bigthm}

We then move on to study higher regularity properties of such solutions under some extra assumptions.  In this context we prove the following :

\begin{bigthm}\label{thmD} Given the initial data $\omega_t,F,f,\f_0$ as above,
assume moreover $F$ to be smooth in $[0,T]\times X\times\mathbb{R}$. Assume also that $f=e^{\psi^+-\psi^-}$ with \begin{itemize}
    \item $\psi^{\pm}$ are quasi-psh functions on $X$, 
    \item $\sup_X\psi^\pm\leq C$, and $\|e^{\psi^-}\|_{L^p}\leq C$ for some constant $C>0$,
    \item $\psi^{\pm}$ are (smooth) locally bounded in a  Zariski open subset $U$ of $\Omega$.
\end{itemize} 
The unique pluripotential solution $\Phi(f,F,\omega_t,\f_0)$ to~\eqref{cmaf}   is smooth on $(0,T)\times U$.
\end{bigthm}

This result can be seen as a generalization of the main result of~\cite{tosatti2015evolution,to2018regularizing}.
It encompasses the case of smooth parabolic Monge-Amp\`ere equations on mildly singular compact hermitian varieties, as well as more degenerate settings, hermitian
 analogues of the main results of~\cite{song2017kahler,boucksom2013regularizing}.

\medskip

We finally apply our results to study the Chern-Ricci flow on mildly singular varieties. 
 We can, in particular, define a good notion of  weak Chern-Ricci flow on varieties with log terminal singularities (and more generally on klt pairs); see Section~\ref{sect: notionCRF} for a precise definition.   Our main results extend to this context as follows:

\begin{bigthm}
\label{thmE} Let $Y$ be a  compact complex variety with log terminal singularities. Assume that $\theta_0$ is a Hermitian metric such that 
    \begin{equation*}
        T_{\rm max}:=\sup \{t>0:\, \exists\; \chi\in\mathcal{C}^\infty(Y) \,\text{such that}\; \theta_0-t\textrm{Ric}(\theta_0)+\dc\chi >0 \}>0.
    \end{equation*} If $S_0=\theta_0+\dc\phi_0$ is a positive (1,1)-current with  bounded potential $\phi_0$, then there exists a unique solution $\omega$ of the weak Chern-Ricci flow~\eqref{crf} starting with $S_0$ for $t\in [0,T_{\rm max})$.
\end{bigthm}

The definition of the weak Chern-Ricci flow is given in Section~\ref{sect: notionCRF}. 
Theorem~\ref{thmE} shows that we can start the Chern-Ricci flow from a positive current with bounded potentials. The weak Chern-Ricci flow smoothes out the initial current in the sense that the flow becomes smooth on the nonsingular part of $Y$ once $t>0$ and the evolving metrics always admit bounded local potentials for any $t\in [0,T_{\rm max})$. In particular, the smoothing property of the Chern-Ricci flow holds when $Y$ is a compact complex manifold (see Section~\ref{sect: smooth} or \cite{tosatti2015evolution, to2018regularizing}).

\subsection*{Organization of paper} In Section~\ref{sect: pp} we construct the class of parabolic potentials and define parabolic complex Monge-Amp\`ere operators. We establish  a priori estimates in Section~\ref{sect: estimate}, which will be used to prove Theorem~\ref{thmA} in Section~\ref{sect: existence}. While, Theorem~\ref{thmB} and Theorem~\ref{thmC} will be proved in Section~\ref{sect: unique_stable}, by establishing uniqueness and stability of pluripotential solutions with time regularity. In Section~\ref{sect: smooth}, we establish the smoothness of the pluripotential solutions, constructed in the previous sections, outside an analytic subvariety under some extra assumptions, proving Theorem~\ref{thmD}. In Section~\ref{sect: crf_lt} we apply these tools to prove the existence and uniqueness for the weak Chern-Ricci flow on compact complex varieties with log terminal singularities, proving Theorem~\ref{thmE}.

\subsection*{Acknowledgement} The author would like to thank his advisors Vincent Guedj and Chinh H. Lu for their help and useful discussions, as well as Tât-Dat Tô for his comments on a first draft.

\section{Preliminaries}
Throughout this article, we let $X$ denote a compact complex manifold of complex dimension $n\geq 1$. We always denote by $\omega_X$ a Hermitian metric on $X$.

\subsection{Recap on elliptic pluripotential theory}
\subsubsection{Big forms} 

We fix $\theta$ a smooth semi-positive (1,1) form on $X$.
Recall that a function  is  quasi-plurisubharmonic function (quasi--psh for short) if it is locally given as the sum of a smooth and a plurisubharmonic function.

\begin{definition}
   A quasi-psh function $u:X\rightarrow [-\infty,+\infty)$ is called $\theta$-plurisubharmonic ($\theta$--psh for short)  if it
   satisfies
$$\theta+\dc u\geq 0$$ in the weak sense of currents. Here $d=\partial+\Bar{\partial}$ and $d^c=\frac{i}{2}(\Bar{\partial}-\partial)$ are both real operators, so that $\dc =i\partial\Bar{\partial}$. 
   We let $\PSH(X,\theta)$ denote the set of all $\theta$-psh  functions which are not identically $-\infty$.
\end{definition}
 
In this paper, we consider the equivalence relation of real (1,1) forms (or currents) on $X$:
\begin{equation*}
    \theta\sim\theta'\Longleftrightarrow\theta=\theta'+\dc \chi \;\;\text{for some function}\;\chi\in\mathcal{C}^\infty(X).
\end{equation*}
We denote by $\{\theta\}$ the {\em equivalence class} of  $\theta$. 

\begin{definition}
  We say that a smooth real (1,1) form $\theta$ is \emph{big} if its equivalence class contains a positive Hermitian current; i.e. there exists a positive (1,1) current $T\in\{\theta\}$ such that $T\geq \delta\omega_X$ for some small constant $\delta>0$.
\end{definition}
It  follows  from  an  approximation  result  of  Demailly~\cite{demailly1992regularization} that  one  can weakly  approximate  a  Hermitian  current  by  the ones  with  analytic singularities. 
\smallskip

A basic example the we are going to consider in Section~\ref{sect: crf_lt} is the following: if $Y$ is a compact complex space endowed with a Hermitian form $\omega_Y$, and $\pi:X\to Y$ is a resolution of singularities, then $\theta=\pi^*\omega_Y$ is big as follows from classical arguments; see e.g.~\cite[Proposition 3.2]{fino2009blow}. Moreover, we can find    a $\theta$-psh function $\rho$ with analytic singularities such that $\theta+\dc\rho\geq \delta\omega_X$, and
\begin{equation*}
    \{\rho>-\infty\}=X\setminus \textrm{Exc}(\pi)=\pi^{-1}(Y_{\rm reg})\simeq Y_{\rm reg}.
\end{equation*}


\subsubsection{Monge-Amp\`ere operators.} Throughout the article, we let $\theta$ denote a smooth real semi-positive and big (1,1) form.

An adaptation of~\cite{bedford1982new} to the Hermitian context allows us to define the complex Monge-Amp\`ere operator $(\theta+\dc u)^n$ for any $\theta$--psh function $u$ which is bounded. We refer the reader to~\cite{dinew2012pluripotential,kolodziej2015weak,guedj2021quasi} for more details.

The mixed Monge-Amp\`ere operator $(\theta_1+\dc u_1)^j\wedge (\theta_2+\dc u_2)^{n-j}$
 are also well-defined for any $0\leq j\leq n$, and any bounded $\theta_i$-psh functions $u_i$, for $i=1,2$; We re call
 the following mixed type inequality which will be used in the sequel.
\begin{lemma}\label{lem: mixed_ineq}
    Let $\theta_1$, $\theta_2$ be a semi-positive (1,1)-forms. Let $u_1\in\PSH(X,\theta_1)\cap L^\infty(X)$ and $u_2\in\PSH(X,\theta_2)\cap L^\infty(X)$ be  such that
    \begin{equation*}
        (\theta_1+\dc u_1)^n\geq e^{f_1} \mu,\quad (\theta_2+\dc u_2)^n\geq e^{f_2}\mu,
    \end{equation*} where $f,f_2$ are bounded measurable functions and $\mu$ is a positive Radon measure with $L^1$ density with respect to Lebesgue measure. Then for any $\delta\in(0,1)$,
    \begin{equation*}
        \left(\delta(\theta_1+\dc u_1)+(1-\delta)(\theta_2+\dc u_2)\right)^n\geq e^{\delta f_1+(1-\delta)f_2}\mu.
    \end{equation*}
\end{lemma}
\begin{proof}
See~\cite[Lemma 1.9]{nguyen2016complex}.
\end{proof}

\subsubsection{Minimum principle}
We will also need the following version of the minimum principle inspired by the one in the local setting; see~\cite[Theorem A]{bedford1976dirichlet}. This somehow generalizes the one established by Kolodziej-Nguyen~\cite[Proposition 2.5]{kolodziej2019stability}, assuming that the  form is merely big. 

Since $\theta$ is big we can find  a $\theta$-psh function $\rho$  with analytic singularities such that  $\theta+\dc\rho\geq \delta\omega_X$ for some small $\delta>0$. We set $\Omega:=\{\rho>-\infty\}$.
	\begin{proposition}
	\label{min-prin}
	    Let $D\subset\subset \Omega$ be a nonempty open set. Fix $u\in\PSH(X,\theta)\cap L^\infty_{\rm loc}(\Omega)$ satisfying $\theta+\dc u\geq \varepsilon_0\omega_X$ for some $\varepsilon_0>0$.  
	    Let $v\in\PSH(X,\theta)\cap L^\infty(X)\cap\mathcal{C}^0(\Omega)$ be such that \begin{equation} (\theta+\dc v)^n\leq c(\theta+\dc u)^n\quad\text{on}\; D,
	    \end{equation} for some $c\in[0,1)$. Then 
	    \begin{equation*}
	        \min_{\bar{D}} (v-u)=\min_{\partial D} (v-u).
	    \end{equation*}
	\end{proposition}
	The proof is the same as that of~\cite[Propostion 2.5]{kolodziej2019stability}, but we shall apply the "comparison principle" established in~\cite[Theorem 1.11]{guedj2021quasi}.
	\begin{proof} 
	Without loss of generality we assume that $\min_{\partial D}(v- u)=0$, i.e. \[u\leq v\quad\text{on}\; \partial D .\] We need to show that $\min_{\bar{D}}(v- u)=0$. Assume by contradiction that 
	$m:=\min_{\bar{D}}(v-u)<0.$ 
	We set \[\phi=\begin{cases}
	v,& \text{on}\; D,\\
	\max ( u,v),&\text{on}\; X\setminus D.
	\end{cases} \] Observe that $\inf_X(\phi-u)\leq m<0$, hence the set
	 $$U(s):=\{\phi<u+\inf_{X}(\phi-u)+s\}\subset\{\phi<u\} \subset\subset D$$ is a nonempty open set for $s>0$ small enough. 
	It follows from~\cite[Theorem 1.11]{guedj2021quasi} that
	\begin{equation*}
	    0<(1-Bs)^n\int_{U(s)}(\theta+\dc u)^n\leq \int_{U(s)}(\theta+\dc v)^n\leq c\int_{U(s)}(\theta+\dc u)^n,
	    \end{equation*}for a uniform constant $B=B(\varepsilon_0,\omega_X)>0$, where we have used the facts that $ (\theta+\dc u)^n\geq \varepsilon_0^n\omega_X^n$
	    and  $(\theta+\dc\phi)^n=(\theta+\dc v)^n$ on the open set $U(s)\subset D$. Therefore, for every $s>0$ small enough,
	    \[ (1-Bs)^n<c \] which is impossible. This completes the proof.
	\end{proof}

\subsubsection{Stability}
We now establish the following $L^1-L^\infty$-stability estimate which allows us to show uniform convergence of solutions as long as they converge in $L^1$.
\begin{theorem}\label{thm: stability} 
  Fix $f_1,f_2\in L^p(X)$ with $p>1$ and $B^{-1}\leq \left(\int_Xf_i^{\frac{1}{n}}dV_X \right)^n\leq \left(\int_Xf_i^p dV_X\right)^{\frac{1}{p}}\leq B$, for some constant $B>1$. Assume $\f_1,\f_2\in\PSH(X,\theta)\cap L^\infty(X)$ satisfy
  \begin{equation*}
      (\theta+\dc\f_i)^n=f_idV_X.
  \end{equation*} Then there exist $\alpha=\alpha(p)>0$ and 
  a constant $C>0$ depending on $n$, $p$, 
  and $B$
  such that
  \begin{equation*}
      \|\f_1-\f_2\|_{L^{\infty}(X)}\leq C\left(\|\f_1-\f_2\|_{1}^\alpha +\|f_1-f_2\|_p \right)^{1/n}.
  \end{equation*}
\end{theorem}
\begin{proof} Without loss of generality we may assume that $\f_1,\f_2\leq 0$.
Set $r=\frac{2p}{p+1}$, hence $1<r<p$.
Since $\f_1$, $\f_2$ are uniformly bounded we have $e^{-\f_i}f_i\in L^r(X)$ for $i=1,2$.  Therefore, there exists a constant $A>0$ depending on  $B$  such that
\begin{equation*}
    A^{-1}\leq \left(\int_Xe^{\frac{-\f_i}{n}}f_i^{\frac{1}{n}}dV_X \right)^n\leq \left(\int_Xe^{-r\f_i}f_i^r dV_X\right)^{\frac{1}{r}}\leq A
\end{equation*}
It follows from~\cite[Theorem 3.5 (1)]{guedj2021quasi} that there is a constant $T>0$ depending on $n$, $r$, $A$ such that
\begin{equation*}\begin{split}
     \|\f_1-\f_2\|_{L^{\infty}(X)}^n&\leq T\|e^{-\f_1}f_1-e^{-\f_2}f_2\|_r\\
     &\leq \|e^{-\f_1}f_1-e^{-\f_2}f_1\|_r +\|e^{-\f_2}f_1-e^{-\f_2}f_2\|_r.
\end{split}
\end{equation*}By H\"older inequality, the latter term is bounded from above by $C(p)\|f_1-f_2\|_p$. Now, we have
\begin{equation*}\begin{split}
    \int_X|f_1|^r|e^{-r\f_1}-e^{-r\f_2}| dV_X&\leq \int_X|f_1|^r r|\f_1-\f_2|e^{-r(\f_1+\f_2)}dV_X\\
    &\leq \left( \int_Xr^p|f_1|^pdV_X\right)^\frac{r}{p}\left(\int_X\left| \f_1-\f_2\right|^{\frac{p}{p-r}}e^{-\frac{pr}{p-r}(\f_1+\f_2)}dV_X \right)^{\frac{p-r}{p}}\\
    &\leq C(p)\|f_1\|^r\|\f_1-\f_2\|_1^{\frac{p-r}{r}},\end{split}
\end{equation*}
where we have used  H\"older inequality again, and the elementary inequality
\begin{equation*}
    |e^a-e^b|\leq |a-b|e^{a+b}, \; \text{for all}\; a,b\geq 0.
\end{equation*} The conclusion thus follows.
\end{proof}

	\subsection{Parabolic potentials}\label{sect: pp}
Parabolic  potentials  form  the  basic  objects  of  our  study.   They  can  be seen as weakly regular family of plurisubharmonic functions.  In this section we define them and recall their first properties. We refer the reader to~\cite{guedj2018pluripotential,guedj2020pluripotential} for more details.
	\subsubsection{Families of quasi-plurisubharmonic functions.} We start with some basic definitions which will be used throughout the paper.
	\begin{definition}\label{def: pp}
	The set of \emph{parabolic potentials}	$\mathcal{P}(X_T,\omega)$ is the set of functions  $\f:(0,T)\times X\to[-\infty,+\infty)$ satisfying
		\begin{itemize}
			\item $x\mapsto\f(t,x)$ is $\omega_t$-psh for all $t\in (0,T)$,
			\item $\f$ is locally uniformly Lipschitz in $(0,T)$.
		\end{itemize}
	\end{definition}
	The last condition means that for any compact sub-interval $J\subset(0,T)$ there exists $\kappa=\kappa_J(\f)>0$ such that for every $x\in X$,
	\begin{align*}
	\f(t,x)\leq \f(s,x)+\kappa|t-s|, \quad\forall\, s,t\in J.
	\end{align*}
	
	A parabolic potential $\f\in\mathcal{P}(X_T,\omega)$ can be extended as a upper semi-continuous function on $[0,T)\times X$ with $\omega_t$-psh slices.
	
	\begin{proposition}
		Let $\f_0$ be a $\omega_0$-psh function. Assume $\f\in\mathcal{P}(X_T,\omega)$ satisfying $\f_t\to \f_0$ in $L^1(X)$ as $t\to 0$. Then the extension $\f:[0,T)\times X\to [-\infty,+\infty)$ is upper semi-continuous. 
	\end{proposition}
	\begin{proof} The proof is almost identical to that of~\cite[Proposition 1.2]{guedj2020pluripotential}. 
		The problem is local, we can thus assume here that $X=B\subset\mathbb{C}^n$ is a neighborhood of $x_0$. Let $\Theta$ be a K\"ahler form in $B$ satisfying $\Theta=\dc h\geq \omega_t$ for all $t$, where $h$ is a local potential on $B$. Changing $\f_t$ by $\f_t+h$ we can assume that $\f_{t}$ are psh and negative in $B$. 
	We can proceed exactly the same as in~\cite[Proposition 2.1]{guedj2020pluripotential} to conclude.
	\end{proof}
	We also have  a compactness result for this class of functions.
	\begin{proposition}
		Let $(\f_j)\subset \mathcal{P}(X_T,\omega)$ be a sequence which
		\begin{itemize}
			\item is locally uniformly bounded from above in $X_T$,
			\item is locally uniformly Lipschitz in $(0,T)$,
			\item does not converge locally uniformly to $-\infty$ in $X_T$.
		\end{itemize}
		Then $(\f_j)$ is bounded in $L^1_{\rm loc}(X_T)$ and there exists a subsequence  which converges to some parabolic potential $\f$ in the $L^1_{\rm loc}(X_T)$-topology.
	\end{proposition}
	\begin{proof}
	  The proof of this result is local and quite close to the classical proof of the analogous result for quasi-psh functions. We refer the reader to~\cite[Proposition 1.14]{guedj2018pluripotential} for more details.  
	\end{proof}
		Fix $\mu$ a Borel measure on X, and let $\ell$ denote the Lebesgue measure on $\mathbb{R}$.
	\begin{lemma}\label{lem: rhs}
	    Fix $\f\in\mathcal{P}(X_T,\omega)$. Then $\partial_t\f(t,x)$ exists for all $(t,x)\in X_T\setminus E$ where $E\subset X_T$ is a $\ell\otimes\mu$-negligible set. 
	    
	    In particular, $\partial_t\f\in L^\infty_{\rm loc}(X_T)$ and $h(\partial_t\f)\ell\otimes\mu$ is a well defined Borel measure on $X_T$ for any continuous function $h:\mathbb{R}\rightarrow\mathbb{R}$.
	\end{lemma}
\begin{proof}
The problem is local. We refer the reader to~\cite[Lemma 1.13]{guedj2018pluripotential} for a proof.
\end{proof}
	
	\subsubsection{Parabolic complex Monge-Amp\`ere operators}
	We assume here that $\f\in\mathcal{P}(X_T,\omega)\cap L^\infty_{}(X_T)$. For each $t\in (0,T)$, the function $x\mapsto\f_t(x)=\f(t,x)$ is $\omega_t$-psh and bounded, hence the wedge product $$(\omega_t+\dc\f_t)^n:=(\omega_t+\dc\f_t)\wedge\cdots\wedge(\omega_t+\dc\f_t)$$ is well-defined as a positive Borel measure on $X$, following the works of Bedford-Taylor \cite{bedford1982new} (see e.g. \cite{dinew2012pluripotential,kolodziej2015weak}).

	\begin{lemma}\label{lem: cln} Let $\f\in\mathcal{P}(X_T)\cap L^\infty_{}(X_T)$ and $\chi$ a continuous test function  on $X_T$. Then the function \[\Gamma_\chi:t\longmapsto\int_X\chi(t,\cdot)(\omega_t+\dc \f_t)^n\] is continuous in $(0,T)$. 
	\end{lemma}
	\begin{proof} We refer to~\cite[Lemma 1.9]{guedj2020pluripotential} for more details.
	\end{proof}
	This lemma allows us to define the $(2n+1)$-current $dt\wedge(\omega_t+\dc\f_T)^n$ on $X_T$.
	\begin{definition}\label{def_lhs}
		Let $\f\in\mathcal{P}(X_T,\omega)\cap L^\infty_{\rm loc}(X_T)$. The map
		\begin{align*}
		\gamma\mapsto\int_{X_T}\chi dt\wedge(\omega_t+\dc\f_t)^n:=\int_0^Tdt\left(\int_X\chi(t,\cdot)(\omega_t+\dc\f_t)^n \right) 
		\end{align*}
		defines a positive $(2n+1)$-current on $X_T$,  denoted by $dt\wedge(\omega_t+\dc\f_t)^n$, which can be identified with a  positive Borel measure on $X_T$.
	\end{definition}
	The operator can also be defined by approximation in the spirit of Bedford and Taylor convergence results \cite{bedford1982new}:

	\begin{proposition}
	    \label{lem: convergence}
		Let $(\f^j)$ be a monotone sequence of functions in $\mathcal{P}(X_T,\omega)$ converging to a function $\f\in \mathcal{P}(X_T,\omega) \cap L^{\infty}_{\loc}(X_T)$ almost everywhere on $X_T$. Then   
		$$dt\wedge(\omega_{t}+\dc\f^j_t)^n\rightarrow dt\wedge(\omega_{t}+\dc\f_t)^n$$
		in the sense of measures on $X_T$.
	\end{proposition}
	\begin{proof} See~\cite[Proposition 1.11]{guedj2020pluripotential}.
	\end{proof}

	    We have the following stronger version of the Chern–Levine–Nirenberg inequalities:

\begin{proposition} Assume $\f^1,\cdots,\f^n\in\mathcal{P}(X_T,\omega)\cap L^\infty_{}(X_T)$. Then there exists a constant $C>0$ such that for any non-positive parabolic potential $\psi\in\mathcal{P}(X_T,\omega)$,
\begin{equation*}
    \begin{split}
        \int_{X_T}-\psi(\omega_t+\dc\f_t^1)\wedge\cdots\wedge(\omega_t+\dc\f_t^n)
        \leq C(\|\psi\|_{L^1(X_T)}+1)\prod_{j=1}^n(\|\f^j\|_{L^\infty(X_T)}+1).
    \end{split}
\end{equation*}
	\end{proposition}
	\begin{proof}
	    In each local chart of triple covers $U_i\Subset V_i\Subset W_i$ one can find a K\"ahler form $\Omega=\dc\rho_i>\omega_t$ for all $t\in [0,T)$. Then we have
	    \begin{equation*}
    \begin{split}
        \int_{(U_i)_T}-\psi(\omega_t+\dc\f_t^1)\wedge\cdots\wedge(\omega_t+\dc\f_t^n)
        \quad\leq \int_{(U_i)_T}-\psi(\dc(\rho_i+\f_t^1))\wedge\cdots\wedge(\dc(\rho_i+\f_t^n)).
    \end{split}
\end{equation*} It follows from the local version of the C-L-N inequalities that the latter integral is bounded by \[ C\int _0^T\left(\|\psi_t\|_{L^1(V_i)}+1\right)\prod_{j=1}^n\left(\|\rho_i+\f^j_t\|_{L^\infty(V_i)}+1\right)dt <+\infty\] and the proof is completed since $\rho_i$ is uniformly under control in $V_i$.
	\end{proof}
	\begin{definition}
	    A family $\mathcal{F}\subset\mathcal{P}(X_T,\omega)$ is uniformly semi-concave in $(0,T)$ if, for any compact  interval $J\Subset(0,T)$, there exists a constant $\kappa=\kappa(\mathcal{F},J)>0$ such that any $\f\in\mathcal{F}$ is  uniformly $\kappa$-concave in $J$, i.e. for each $x\in X$, $t\mapsto\f(t,x)-\kappa t^2$ is concave in $t\in J$.
	\end{definition}
	
	Fix $\mu$ a Borel measure on X, and let $\ell$ denote the Lebesgue measure on $\mathbb{R}$.
	\begin{theorem}
	    Let $(f_j)$ be a sequence of positive functions which converges to $f$ in $L^1(X_T,\ell\otimes\mu)$. Let $(\f^j) \subset\mathcal{P}(X_T,\omega)$ be a sequence of parabolic potentials which \begin{itemize}
	        \item converges $\ell\otimes\mu$-almost everywhere in $X_T$ to a function $\f\in \mathcal{P}(X_T,\omega)$,
	        \item is uniformly semi-concave in $(0,T)$.
	    \end{itemize}
	    Then $\lim_{j\to+\infty}\dot{\f}^j(t,x)=\dot{\f}(t,x)$ for $\ell\otimes\mu$-almost every $(t,x)\in X_T $, and \begin{equation*}
	        h(\dot{\f}^j)f^j\ell\otimes\mu\longrightarrow h(\dot{\f})f\ell\otimes\mu,
	    \end{equation*} in the weak senses of measures on $X_T,$ for all $h\in\mathcal{C}^0(\mathbb{R},\mathbb{R})$.
	\end{theorem}
	We refer the reader to \cite[Theorem 1.14]{guedj2020pluripotential} for a proof.

\subsection{Assumptions and Notations}\label{assumption}

In the whole article,  we let $X$ denote a compact complex manifold of complex dimension $n\geq 1$. We always denote by $\omega_X$ a Hermitian metric on $X$.  

We fix $T\in (0,+\infty]$. We are mainly concerned with finite time intervals, i.e. $T<+\infty$, and we implicitly assume that our data are possibly defined in slightly large time interval, i.e. on $(0,T+\varepsilon)$ for some $\varepsilon>0$. 
We let $X_T$ denote the real $(2n+1)$-dimensional manifold $X_T=(0,T)\times X$ with parabolic boundary 
\[\partial X_T:=\{0\}\times X. \] 
We fix a smooth semi-positive  $(1,1)$ form $\theta$ whose equivalence class is big, i.e. contains a positive (1,1) current  which dominates  a hermitian form. We fix $\rho$ a $\theta$--psh function  with analytic singularities such that
\begin{equation}\label{funct: rho}
    \theta+\dc\rho\geq \delta \omega_X\quad\text{for some }\;\delta>0.
\end{equation} We let $\Omega$ denote the {\em positive locus} of the big form $\theta$, \begin{equation}\label{ample locus}
  \Omega:=X\setminus\{\rho=-\infty\}.
\end{equation}
\subsubsection{Assumptions on the forms}\label{sect: forms}

We assume throughout the article that $(\omega_t)_{t\in[0,T)}$ is a smooth family of semi-positive (1,1) forms on $X$ satisfying
\begin{equation*}
    \theta\leq \omega_t
\end{equation*} for all $t\in[0,T)$. For finite times we can also assume, up to multiplying by a positive constant, that $\omega_t\leq \omega_X$ for all $t\in[0,T)$.

In Section~\ref{sect: estimate}, we need to assume that $t\mapsto\omega_t$  moreover satisfies for $t\in[0,T)$,
\begin{equation}\label{assump: om}
    -A\omega_t\leq\dot{\omega}_t\leq A\omega_t,
\end{equation}
and \begin{equation}
    \ddot{\omega}_t\leq A\omega_t
\end{equation} for some constant $A>0$. The lower bound~\eqref{assump: om} is equivalent to the fact that $t\mapsto e^{At}\omega_t$ is increasing. In particular, 
\[ \omega_{t+s}\geq e^{-As}\omega_t\geq (1-As)\omega_t,\quad s>0. \] The latter will be used on several occasions in  the sequel.
    %

\subsubsection{Assumptions on the densities} We assume throughout the article that
\begin{itemize}
    \item $dV$ is a fixed 
volume form on $X$;
\item $0\leq f\in L^p(X,dV)$ for some $p>1$, and $\Vol(\{f=0\})=0$;
\item $(t,x,r)\mapsto F(t,x,r)$ is a continuous function on $[0,T)\times X\times \mathbb{R}$;
\item $r\mapsto F(\cdot,\cdot,r)$ is uniformly quasi-increasing, i.e. there exists a constant $\lambda_F\geq 0$ such that for every $(t,x)\in[0,T)\times X$, the function
\begin{equation}\label{assump: increasing}
    r\mapsto F(t,x,r)+\lambda_F r\quad \text{is increasing in}\; \mathbb{R};
\end{equation}
\item $(t,r)\mapsto F(t,\cdot,r)$ is locally uniformly Lipschitz, i.e. for any $J\Subset [0,T)\times\mathbb{R}$ there exists a constant $\kappa_J>0$ such that for every $x\in X$, $(t,r), (t',r')\in J$,
\begin{equation}\label{assump: Lips}
    |F(t,x,r)-F(t',x,r')|\leq\kappa_J(|t-t'|+|r-r'|);
\end{equation}
\item $(t,r)\mapsto F(t,\cdot,r)$ is locally uniformly semi-convex, i.e. for any $J\Subset [0,T)\times\mathbb{R}$ there exists a constant $C_{J}>0$ such that for every $x\in X$,
\begin{equation}\label{assump: convex}
    (t,r)\mapsto F(t,x,r)+C_J(t^2+r^2) \quad\text{is convex in}\; J.
\end{equation}
\end{itemize}
Note that if $F$ is $\mathcal{C}^2$-smooth then the local assumptions~\eqref{assump: Lips} and~\eqref{assump: convex} are automatically satisfied, while~\eqref{assump: increasing} is a global assumption.
  
 Our  assumptions  on  the  data $F$, $f$ are  mild  enough  so  that the results of this article can be applied to the study of the Chern-Ricci flow on mildly singular Hermitian varieties. We refer to Sect.~\ref{sect: crf_lt} for more details and geometric applications.

	\section{A priori estimates}\label{sect: estimate}
	In this section we assume that $\f_t=\f(t,\cdot)$ is a smooth solution to~\eqref{cmaf} with given smooth data ($\omega_t$, $F$, $f$, $\f_0$), i.e. $t\mapsto\omega_t$ is a smooth family of Hermitian forms, $F$ and $f$ are smooth densities with $f$ being strictly positive, and $\f_0$ is  smooth and strictly $\omega_0$-psh. 
	\smallskip
	
	Our aim is to establish various a priori estimates that will allow us to construct weak solutions to the corresponding degenerate equations.  
	We will make several extra assumptions, depending on the a priori estimates that we are interested in.
	\subsection{Bounding the oscillation of $\f_t$}
	Recall that $\omega_t\geq \theta$ for all $t\in [0,T]$, where $\theta$ is a semi-positive and big (1,1) form. For finite time we can assume without loss of generality that $\omega_t\leq \omega_X$ for $t\in[0,T]$.
	It follows from~\cite[Theorem 3.4]{guedj2021quasi} (respectively~\cite[Theorem 5.8]{kolodziej2015weak}) that there exist a constant $c_-$ and a bounded $\theta$-psh function $\phi$ (respectively a constant $c_+$ and a bounded $\omega_X$-psh function $\Phi$) such that
	\begin{align*}
	(\theta+\dc \phi)^n=e^{c_-} f dV, \qquad (\omega_X+\dc\Phi)^n=e^{c_+}fdV.
	\end{align*}
	 Up to adding a constant we may assume that
	\[\sup_X\phi=\inf_X\Phi=0. \] 
	\begin{proposition}\label{prop_0dot}
		There exists a uniform constant $C_0>0$ only depending on  $\sup_X|\f_0|$, $\lambda_F$, $\inf_X\phi$, $\sup_X\Phi$, and $\sup_{X_T}F(t,x,0)$ such that
		\begin{align*}
		|\f_t(x)|\leq C_0, \quad\forall\, (t,x)\in [0,T]\times X.
		\end{align*}
	\end{proposition}
	Recall that $\lambda_F\geq 0$ is a constant    such that, for all $(t,x)\in X_T$, the function $F(t,x,r)+\lambda_F r$ is increasing on $\mathbb{R}$.
	\begin{proof} 
		Set, for any $t\in\mathbb{R}$,
		\begin{align*}
		\gamma(t):=\sup_X|\f_0|e^{\lambda_Ft}+C\frac{e^{\lambda_F t}-1}{\lambda_F},
		\end{align*}
		for $C>0$ to be chosen later. It is clear that $\gamma(0)=\sup_X|\f_0|$ and $\gamma$ satisfies the following ordinary differential equation $\gamma'(t)-\lambda_F\gamma(t)=C$.
		
		We first get a lower bound for $\f_t$. Set, for $(t,x)\in [0,T]\times X$,
		\[u(t,x):=\phi(x)-\gamma(t). \] We observe that $u_t$ is hence $\theta$-psh such that $u_0\leq \f_0$, and 
		\begin{align*}
		(\omega_t+\dc u_t)^n\geq (\theta+\dc u_t)^n=e^{c_-}fdV.
		\end{align*}
		On the other hand, since $r\mapsto F(\cdot,\cdot,r)+\lambda_Fr$ is increasing and $\phi\leq 0$, we have for $(t,x)\in [0,T]\times X$,
		\begin{align*}
		\dot{u}_t(x)+F(t,x,u_t(x))&= -\gamma'(t)+F(t,x,u_t(x))\\
		&\leq -\gamma'(t)+F(t,x,0)-\lambda_F(\phi(x)-\gamma(t))\\
		&\leq \lambda_F|\phi(x)|-C+F(t,x,0)\\
		&\leq c_-,
		\end{align*}
		choosing $C>0$ large enough. Therefore
		\begin{align*}
		(\omega_t+\dc u_t)^n\geq e^{\dot{u}_t+F(t,\cdot,u_t(\cdot))}fdV.
		\end{align*}
		Now the maximum principle ensures that $\f\geq u $ on $[0,T]\times X$.
		
	For a upper bound,	
	set for any $(t,x)\in [0,T]\times X$,
		$v(t,x):=\Phi(x)+\gamma(t)$, where $\gamma(t)$ is the solution to ODE: $\gamma'(t)-\lambda_F\gamma(t)=C$ with $C=c_+-\inf_{X_T}F(t,x,0)$, and $\gamma(0)=\sup_X\f_0$. We can check that
		\begin{align*}
		e^{\dot{v}_t+F(t,x,v_t)}fdV\geq e^{c_+}fdV= (\omega_X+\dc v_t)^n
		\end{align*}
		Since $v_0\geq \f_0$ it follows the maximum principle that $v(t,x)\geq \f(t,x)$ for any $(t,x)\in X_T$ which implies the upper bound for $\f$. More precisely, we have for any $(t,x)\in[0,T]\times X$,
		\begin{align*}
		|\f_t(x)|\leq C_0:=\sup_X|\f_0|e^{\lambda_Ft}+C\frac{e^{\lambda_F t}-1}{\lambda_F}
		\end{align*}
		where $C$ is the following uniform constant
		\begin{align*}
		C=\sup_{X_T}|F(t,x,0)|+(\lambda_F+1)\sup_X(|\phi|+|\Phi|)+\max (-c_-,c_+). 
		\end{align*}
	\end{proof}

		Let $\alpha>0$ be such that $\alpha T<1$. The following barrier constructions  will be useful in showing that the pluripotential solution to~\eqref{cmaf} has the right value at $t=0$:
	\begin{proposition}\label{prop_barrier}
	  There exists a uniform constant $C>0$ such that for all $(t,x)\in [0,T)\times X$,
		\begin{align*}
		\f_t(x)\geq (1-\alpha t)e^{-At}\f_0(x)+\alpha t\rho+n(t\log\alpha t-t)-C\frac{e^{\lambda_F t}-1}{\lambda_F}.
		\end{align*}
	\end{proposition}
		
	\begin{proof} Recall  that  $\dot{\omega}_t\geq -A\omega_t$ for some constant $A>0$. In particular $\omega_t\geq e^{-At}\omega_0$ and $\omega_t\geq \theta$.
		Consider for any $(t,x)\in [0,T)\times X$,
		\begin{align*}
		u(t,x):=(1-\alpha t)e^{-At}\f_0+\alpha t\phi_-+n(t\log\alpha t-t)-C\frac{e^{\lambda_Ft}-1}{\lambda_F}
		\end{align*}
		for $C>0$ a uniform constant. 
		We proceed the same as in~\cite[Proposition 2.2]{guedj2020pluripotential} to show that
		\begin{align*}
		(\omega_t+\dc u_t)^n\geq e^{\dot{u}_t+F(t,\cdot,u_t)}fdV,
		\end{align*}
		and moreover $u_0\leq \f_0$.
		Therefore the maximum principle yields the desired estimate. 
	\end{proof}
	
	\subsection{Lipschitz control in time}
	Our goal in this section is to establish an a priori bound which allows us to show that the solutions $\f$ to degenerate complex Monge-Amp\`ere flows~\eqref{cmaf} are locally uniformly Lipschitz in time, away from zero.

	We recall here that there exists a constant $A>0$ such that
	\begin{equation}
	    -A\omega_t\leq \dot{\omega}_t\leq A\omega_t\;\, \forall \, t\in [0,T].
	\end{equation}
	Recall also that $F(t,x,r)$ is quasi-increasing in $r$ i.e. there exists a constant $\lambda_F>0$ such that for every $(t,x)\in [0,T)\times X$, the function $r\mapsto F(t,x,r)+\lambda_F r$ is increasing in $\mathbb{R}$.
	\begin{proposition}\label{prop_1dot}
		There exists a uniform constant $C>0$ such that for all $(t,x)\in X_T$,
		\begin{align*}
	 n\log t-C	\leq \dot{\f}_t(x)\leq \frac{C}{t}.
		\end{align*}
	\end{proposition}
	
	\begin{proof} 
	The proof is identical to that of~\cite[Theorem 2.4, 2.5]{guedj2020pluripotential}.
	Note here that $C$ depends explicitly on  $T$, $C_0$ (defined in Proposition~\ref{prop_0dot}), and upper bounds for $\|f\|_{L^p}$, 
	\[ \sup_{[0,T)\times X\times [-C_0\times C_0]}\frac{\partial F}{\partial t},\; \text{and}\; \sup_{[0,T)\times X\times [-C_0\times C_0]}\frac{\partial F}{\partial r}.\]

	\end{proof}
	\subsection{Semi-concavity in time} We now establish that $\omega_t$-psh solutions $\f_t$ to~\eqref{cmaf} are semi-concave in time away from zero.
	
	We assume in this subsection that there exists a constant $A>0$ such that, for all $t\in [0,T]$, 
	\begin{equation}\label{ass: omega}
	    -A\omega_t\leq\dot{\omega}_t\leq A\omega_t,\;\, \text{and}\,\; \ddot{\omega}_t\leq A\omega_t.
	\end{equation}
	We also assume that $(t,r)\mapsto F(t,x,r)$ is uniformly semi-convex, i.e. there exists a constant $C_F>0$ such that for every $x$, the function $$(t,x)\mapsto F(t,x,r)+C_F(t^2+r^2) \;\text{is convex on}\; [0,T]\times [-C_0,C_0].$$

	\begin{proposition}\label{prop_2dot}
		There exists a uniform constant $C>0$ such that for all $(t,x)\in (0,T]$,
		\begin{align*}
		\ddot{\f}_t(x)\leq \frac{C}{t^2},
		\end{align*} where $C$ depends explicitly on $T$, $\|f\|_{L^p}$, $C_F$, $\sup_{[0,T]\times X}|\f|$, $\sup_{[0,T]\times X}t|\dot{\f}|$, and the $L^\infty$-norms of $\partial F/\partial r$, $\partial F/\partial t$.
	\end{proposition}

	\begin{proof}
The proof is almost identical to that of~\cite[Theorem 2.7, 2.9]{guedj2020pluripotential}. We mention here that Lemma~2.8 in~\cite{guedj2020pluripotential} still holds when $\omega$ is Hermitian and $\eta$ is no longer closed.
	\end{proof}

	\begin{theorem}\label{thm_211}
		Let $(f_j)_{j\in\mathbb{N}}$ be a sequence of $L^p(X)$-densities converging towards $f$ in $L^1(X)$. Let $F_j(t,x,r)$ be continuous densities which uniformly converge towards  $F$. Let $\f_j(t,x)$ be a family of $\omega_t$-psh functions such that
		\begin{itemize}
			\item $\f_j$ are uniformly bounded,
			
			\item for any $x\in X$, $\ddot{\f}_j(t,x)\leq C/t^2$ for some uniform constant $C>0$.
		\end{itemize}
		Then there exists a bounded function $\f\in\mathcal{P}(X_T,\omega)$ such that, up to passing a subsequence, $\f_j\to \f$ in $L^1_{\rm loc}(X_T)$, 
		and
		\begin{align*}
		e^{\dot{\f}_j+F_j(t,x,\f_j(t,x))}f_j(x)dV(x)\wedge dt\rightarrow e^{\dot{\f}+F(t,x,\f(t,x))}f(x)dV(x)\wedge dt
		\end{align*}
		in the weak sense of measures on $X_T$.
	\end{theorem}
	\begin{proof} We refer to~\cite[Theorem 2.11]{guedj2020pluripotential} for more details.
		
		
	\end{proof}
	\subsection{Stability estimates}
	\begin{proposition}\label{stab1}
		Let  $0<\varepsilon<T'<T$. 
	Assume that $\f^1,\f^2\in\mathcal{P}(X_T,\omega)\cap \mathcal{C}^{\infty}(X_T)$  solve
		\begin{align*}
		dt\wedge (\omega_t+\dc \f_t^i)^n=e^{\dot{\f}_t^i+F_i(t,\cdot,\f_t^i)}f_idV\wedge dt,
		\end{align*}
		with assumptions that $(F_i,f_i)$ are smooth for $i=1,2$. 
		Then for all $(t,x)\in [\varepsilon,T']\times X$,
		\begin{align*}
		|\f^1(t,x)-\f^2(t,x)|\leq B\left(\|\f^1-\f^2\|_{L^1(X_T)}^\gamma+\|f_1-f_2\|_{L^p(X)}\right)^{ \frac{1}{n}},
		\end{align*}
		where $0<\gamma=\gamma(n,p)$ while $0<B$  depends on $\varepsilon$, $T'$, $\theta$, $\omega_X$, and  upper bounds for $\|f_i\|_{L^{\frac{1}{n}}}$, $\|f_i\|_{L^p}$, $\|\f^i\|_{L^\infty(X_T)}$, and $\sup_{[\varepsilon,T]\times X}|\dot{\f}^i_t|$.
	\end{proposition}
	\begin{proof}
		We are going to use the stability result established in~\cite[Theorem 2.3]{guedj2021quasi}.   
		For any $t\in [\varepsilon,T']$, consider the densities \[f^i_{t}=e^{\dot{\f}^i_t+F_i(t,\cdot,\f_t^i)}f_i,\quad i=1,2. \] 
		It follows from Proposition~\ref{prop_1dot} that for any $t\in[\varepsilon,T']$, $\dot{\f}_t^i$ are uniformly bounded by $C_1/\varepsilon$, while Proposition~\ref{prop_0dot} ensures that the $\f_t^i$ are uniformly bounded. We thus deduce that the $L^p$-norms of  the densities $f^i_{t}$ are uniformly bounded from above by $$A_1:=e^{C_1/\varepsilon+M_F}\left(\|f_1\|_{L^p}+\|f_2\|_{L^p}\right).$$ 
		Similarly, the $L^{\frac{1}{n}}$-norms of the densities $f_t^i$ are uniformly bounded from below by \begin{equation*}
		    A_2:=e^{-C_1/\varepsilon-M_F}\left(\|f_1\|_{L^\frac{1}{n}}+\|f_2\|_{L^\frac{1}{n}}\right)
		\end{equation*}
		for $t\in[\varepsilon,T']$.		
		 Theorem~\ref{thm: stability} ensures that
		\begin{align*}
		\|\f^1_t-\f^2_t\|_{L^\infty(X)}\leq C\left(\|\f^1_t-\f^2_t\|^\gamma_{L^1(X)}+\|f^1_t -f^2_t\|_{L^p(X)}\right)^{\frac{1}{n}},\quad \forall t\in [\varepsilon,T'],
		\end{align*}  
		 where $\gamma\in (0,1)$ only depends on $p,n$, and $C$ only depends on $n$, $p$, upper bounds for $\|\f^i\|_{L^\infty(X_T)}$, $\|f_t^i\|_{L^\frac{1}{n}}$ and $\|f_t^i\|_{L^p}$. On the other hand, it follows from \cite[Lemma 1.5]{guedj2020pluripotential} that
		\begin{align*}
		\|\f^1_t-\f^2_t\|_{L^1(X)}\leq B\max\{\|\f^1-\f^2\|_{L^1(X_T)}, \|\f^1-\f^2\|_{L^1(X_T)}^{1/2} \},
		\end{align*}
		with $B=2\max\{\sqrt{\kappa},(T-T')^{-1} \}$, where $\kappa$ is a uniformly Lipschitz constant of $\f^1-\f^2$ in $[\varepsilon,T']$. The proof follows from the last two inequalities. 
	\end{proof}
	The previous results yields easily the following useful information:
	\begin{corollary}
	Assume that $\|f_j\|_{L^p}$, $\|F_j\|_{L^\infty},\|\partial_t F_j\|_{L^\infty}$ and $\|\partial_rF_j\|_{L^\infty}$ are uniformly bounded. If a sequence $(\f^j)$ of solutions to $\eqref{cmaf}_{f_j,F_j}$ converges in $L^1(X_T)$ to $\f$, then it converges uniformly on any compact subsets of $(0,T)\times X$.  
	\end{corollary}
	\section{Pluripotential solutions}
	
	From now on we assume that the family $(\omega_t)$ and the density $F$ and $f$ satisfy the conditions given in the introduction.
	
	For bounded parabolic potentials $\f\in\mathcal{P}(X_T,\omega)\cap L^\infty(X_T)$, we can define the parabolic complex Monge-Amp\`ere equation in the sense of measures on $(0,T)\times X$:
	\begin{equation}\label{eq cmaf}\tag{CMAF}
	    (\omega_t+\dc\f_t)^n\wedge dt=e^{\dot{\f}_t+F(t,x,\f_t)}f(x)dV(x)\wedge dt.
	\end{equation}
	Indeed, it follows from Definition~\ref{def_lhs} that the left-hand side is a well-defined Radon positive measure, while Lemma~\ref{lem: rhs} ensures that so is the right-hand side.
	\subsection{Existence results}\label{sect: existence}
	Our goal in this section is to run the flow \eqref{eq cmaf} from a singular initial data $\f_0\in \PSH(X,\omega_0)$ having uniform bounds. The strategy is to approximate $\f_0$ by a  decreasing sequence of  smooth $\omega_0$-psh functions, $f$ by a smooth sequence functions $f_j$ in $L^p(dV)$, and $F$ by a smooth sequence functions $F_j$, we then consider unique solutions $\f_{t}^j$ to the flow \eqref{eq cmaf} with smooth data $\f_{0,j}$, $F_j$, $f_j$. In order to pass to the limit we shall use the a priori estimates established in the previous section.  
	
	Before proving Theorem~\ref{thmA} in the introduction, we give the following definition of pluripotential solutions to the Cauchy problem.
	\begin{definition}
	    A parabolic potential $\f\in\mathcal{P}(X_T,\omega)$ is a {\em pluripotential solution} to~\eqref{eq cmaf} with initial data $\f_0\in\PSH(X,\omega_0)\cap L^\infty(X)$ if $\f$ satisfies~\eqref{eq cmaf} in the sense of measures on $X_T$ and $\f_t\to \f_0$ in $L^1(X)$ as $t\to 0^+$. 
	\end{definition}
	\begin{theorem}\label{thm: existence}
		Let $\f_0$ be a bounded $\omega_0$-psh function on $X$. Assume that $(\omega,F,f)$ is as in the introduction. Then there exists $\f\in\mathcal{P}(X_T,\omega)$ solving \eqref{eq cmaf} with initial data $\f_0$ such that for all $0<T'<T$,
		\begin{itemize}
			\item $(t,x)\mapsto\f(t,x)$ is uniformly bounded in $(0,T']\times X$,
			
			\item $t\mapsto \f(t,\cdot)-n(t\log t-t)+Ct$ is increasing on $(0,T')$,
			
			\item $t\mapsto\f(t,\cdot)-Ct^2$ is concave on each compact subset of $(0,T')$ for some $C>0$,
			
			\item  $\f_t\to\f_0$ as $t\to 0^+$ in $L^1(X)$.
		\end{itemize}
	\end{theorem}
	\begin{proof}
		We first approximate
		\begin{itemize}
			\item $F$ by smooth densities $F_j$ with uniform constants $L_{F_j},C_{F_j},\lambda_{F_j}$;
			
			\item $f$ by smooth densities $f_j>0$ in $L^p(X)$;
			
			\item $\f_0$ by smoothly decreasing $\omega_0$-psh functions $\f_{0,j}$ thanks to \cite{demailly1992regularization, blocki2007regularization};
			
			\item $\omega_t$ by smooth (in $t$) $\omega_t^j=\omega_t+2^j\omega_X$ and $\omega_t^j$ satisfies the assumptions in the introduction.
		\end{itemize}
		It is well-known  (see e.g. \cite[Theorem 1.2]{tosatti2015evolution}) that for every $j$, there exists a unique smooth function $\f^j\in\mathcal{P}(X_T,\omega)$ to $\eqref{eq cmaf}_{F_j,f_j}$, i.e.
		\begin{align*}
		dt\wedge(\omega_t^j+\dc\f_t^j)^n=e^{\dot{\f}^j(t,x)+F_j(t,x,\f^j(t,x))}f_j dt\wedge dV(x).
		\end{align*}
		By the previous section, we see that the $\f^j$'s are uniformly bounded and the derivatives $\ddot{\f}^j$ are locally uniformly bounded from above in $X_T$. Extracting and relabelling, it follows from Theorem \ref{thm_211} that there exists  $\f\in \mathcal{P}(X_T,\omega)\cap L^\infty_{\rm loc}(X_T)$ such that $\f^j\to \f$ in $L^1_{\rm loc}(X_T)$ and 
		\begin{align*}
		e^{\dot{\f}_j+F_j(t,x,\f_j(t,x))}f_j(x)dV(x)\wedge dt\rightarrow e^{\dot{\f}+F(t,x,\f_j(t,x))}f(x)dV(x)\wedge dt
		\end{align*}
		in the weak sense of measures on $X_T$. 
		
		We claim that $\f^j\to \f$ locally uniformly in $X_T$. Indeed, fix $0<\varepsilon<T'<T$. Since the densities $f_j$ have uniform $L^p$ norms, it follows from Proposition~\ref{prop_1dot} that for each $j$, $\dot{\f}^j$ is uniformly bounded on $[\varepsilon,T']\times X$. By Proposition \ref{stab1}, we have
		\begin{align*}
		|\f^j(t,x)-\f^k(t,x)|\leq B\left(\|\f^j-\f^k\|_{L^1(X_T)}^\alpha+\|f_j-f_k\|_p \right)^{1/n},
		\end{align*}
		where $B>0$ and $\alpha\in (0,1)$ are uniform constants which do not depend on $j,k$, and $t\in [\varepsilon,T']$. The claim thus follows.
		
		Therefore
		\begin{align*}
		dt\wedge(\omega_t+\dc\f_t^j)^n\rightarrow dt\wedge(\omega_t+\dc\f_t)^n
		\end{align*} in the sense of measures on $X_T$, hence $\f$ solves~\eqref{eq cmaf}. 
		
		We now show similarly that $\f$ is locally uniformly semi-concave in $(0,T)$. Since the densities $f_j$ are uniformly bounded in $L^p(X)$, Proposition \ref{prop_2dot} ensures that there is a uniform constant $C>0$ such that for all $j$,
		\begin{align*}
		\ddot{\f}^j(t,x)\leq C/t^2,\quad\forall (t,x)\in X_T. 
		\end{align*}
		Hence, for each compact sub-interval $J\subset\subset(0,T)$ there exists a constant $C=C_J>0$ such that the functions $t\mapsto\f^j(t,x)-Ct^2$ are concave in $J$ for all $x\in X$. The same properties hold for the limiting function $\f(t,x)$ by letting $j\to+\infty$. 
		
		The last statement follows from Lemma \ref{lem_aver} below.
	\end{proof}
	\begin{lemma}\label{lem_aver} Let $\f\in\mathcal{P}(X_T)\cap L^\infty_{\rm loc}(X_T)$ be a solution to \eqref{eq cmaf}. Then
		there exists a uniform constant $C>0$ such that
		\begin{align*}
		\int_X\f_td\mu\leq \int_X\f_0d\mu+C.
		\end{align*}
	\end{lemma}
	\begin{proof}
		Set $m:=\inf_{X_T\times [-C_0,C_0]}F(t,x,r)$. We first observe that
		\begin{align*}
		\int_Xe^{\dot{\f}_t+m}d\mu\leq \int_X(\omega_t+\dc\f_t)^n\leq M<+\infty,
		\end{align*} where the last inequality follows from~\cite[Proposition 2.3]{dinew2012pluripotential}. Note that we have shown that $|\f(t,x)|\leq C_0$ in Proposition~\ref{prop_0dot}.
		It follows from Jensen's inequality that
		\begin{align*}
		\int_X e^{\dot{\f}_t}\frac{d\mu}{\mu(X)}\geq \exp\left(\int_X\dot{\f}_t\frac{d\mu}{\mu(X)} \right).
		\end{align*}
		Combining these two estimates we obtain
		\begin{align*}
		\int_X\dot{\f}_td\mu\leq C:=m\mu(X)+\mu(X)\log M-\mu(X)\log\mu(X).
		\end{align*}
		We then infer that the function $t\mapsto \int_X\f_td\mu-Ct$ is decreasing in $t\in(0,T)$, hence
		\begin{align*}
		\int_X\f_td\mu\leq \int_X\f_0d\mu+C,
		\end{align*} as required.
	\end{proof}
\subsection{Pluripotential subsolutions/supersolutions}
Before establishing a pluripotential parabolic comparison principle, we introduce in this section some essential materials. \subsubsection{Definitions}	 
\begin{definition}
	A parabolic potential $\f\in\mathcal{P}(X_T)\cap L^\infty(X_T)$ is called a \emph{pluripotential subsolution} of~\eqref{eq cmaf} if 
	\begin{align*}
	dt\wedge(\omega_t+\dc\f_t)^n\geq e^{\dot{\f}_t+F(t,\cdot,\f_t)}fdV\wedge dt
	\end{align*}
	holds in the sense of measures in $(0,T)\times X$.
	
	A parabolic potential $\f\in\mathcal{P}(X_T)\cap L^\infty(X_T)$ is called a \emph{pluripotential supersolution} of~\eqref{eq cmaf} if 
	\begin{align*}
	dt\wedge(\omega_t+\dc\f_t)^n\leq e^{\dot{\f}_t+F(t,\cdot,\f_t)}fdV\wedge dt
	\end{align*}
	holds in the sense of measures in $(0,T)\times X$.
\end{definition}
In the sequel one can interpret these notions by considering a family of inequalities on slices:
\begin{lemma}\label{lem_subsol}
	Let $\f\in\mathcal{P}(X_T)\cap L^\infty(X_T)$. Then
	
	1) $\f$ is a pluripotential subsolution of \eqref{eq cmaf} if and only if for all $t\in (0,T)$,
	\begin{align*}
	(\omega_t+\dc\f_t)^n\geq e^{\dot{\f}_t+F(t,\cdot,\f_t)}fdV
	\end{align*}
	in the sense of measures in $X$.
	
	2) $\f$ is a pluripotential supersolution of \eqref{eq cmaf} if and only if for all $t\in (0,T)$,
	\begin{align*}
	(\omega_t+\dc\f_t)^n\leq e^{\dot{\f}_t+F(t,\cdot,\f_t)}fdV
	\end{align*}
	in the sense of measures in $X$.
\end{lemma}
\begin{proof}
	\cite[Lemma 3.11]{guedj2020pluripotential}
\end{proof}

We use properties of solutions to complex Monge-Amp\`ere equations to show that parabolic supersolutions automatically have continuity properties: 
\begin{proposition}\label{prop: continuous}
    Let $\psi\in\mathcal{P}(X_T,\omega)\cap L^\infty(X_T)$ be a pluripotential supersolution to \eqref{eq cmaf}. Then $\psi$ is continuous on $(0,T)\times \Omega$.
\end{proposition}
Recall here that $\Omega$ is defined by $\Omega:=\{\rho>-\infty\}$, where $\rho$ is an $\theta$-psh function with analytic singularities such that $\theta+\dc\rho$ dominates a Hermitian form. 
\begin{proof}
    Fix $J\Subset (0,T)$. By definition, for every $t\in (0,T)$ we have
    \begin{equation*}
        (\omega_t+\dc\psi_t)^n\leq e^{\dot{\psi}_t+F(t,\cdot,\psi_t)}fdV
    \end{equation*} in the sense of measures in $X$. 
    
    Since $\psi$ is locally uniformly Lipschitz in $t$ and $F$ is bounded from above, there exists a constant $M=M(J)>0$ for almost every $t\in J$. We thus obtain
    \begin{equation*}
       (\omega_t+\dc\psi_t)^n\leq e^M fdV 
    \end{equation*} for almost every $t\in J$. This inequality actually holds for all $t\in J$ by the (weak) continuity of the left-hand side. It follows from~\cite[Theorem 3.7]{guedj2021quasi} that $\psi_t$ is continuous on $\Omega$ for each $t\in J$. Let us emphasize that in the proof of~\cite[Theorem 3.7, p. 25]{guedj2021quasi} we only  require $\omega_t$ to be merely big since the supersolution $\psi_t$ is bounded on $X$, for each $t>0$.
    Since $\psi$ is uniformly Lipschitz in $J$ it follows that $\psi$ is continuous on $J\times \Omega$ as follows from standard arguments; see e.g.~\cite[Proposition~3.12]{guedj2020pluripotential}.
\end{proof}
\subsubsection{Regularization of subsolutions.} In this part, we regularize subsolutions of~\eqref{eq cmaf} in time. This is analogous to that of~\cite[Sect. 3.4.3]{guedj2020pluripotential}.
We always assume that for all $t\in [0,T]$,
\begin{align}\label{hyp_1}
-A\omega_t\leq \dot{\omega}_t\leq A\omega_t,
\end{align}
for some constant $A>0$. We briefly recall a regularization process for subsolutions. Fix $0<T'<T$ and $\varepsilon_0>0$ such that $(1+\varepsilon_0)T'\leq T$. It follows from \eqref{hyp_1} that there exists $A_1>0$ such that for all $t\in (0,T')$ and $|s-1|<\varepsilon_0$,
\begin{align*}
\omega_t\geq (1-A_1|s-1|)\omega_{ts}.
\end{align*}
For $|s-1|<\varepsilon_0$ we set
\begin{align*}
\lambda_s:=\frac{|s-1|}{s},\quad\alpha_s=s(1-\lambda_s)(1-A_1|s-1|)\in (0,1).
\end{align*}
Up to shrinking $\varepsilon_0$, we can also assume that 
\begin{align*}
\lambda_s:=\frac{\lambda_s}{1-\alpha_s}\geq \varepsilon_1>0.
\end{align*}
It follows from \cite[Theorem 3.4]{guedj2021quasi} that there exist an $\varepsilon_1\theta$-psh function $\rho_1$ and a constant $c_1\in\mathbb{R}$ satisfying
\begin{align*}
(\varepsilon_1\theta+\dc\rho_1)^n=e^{c_1}fdV.
\end{align*}
This solution can be normalized by $\sup_X\rho_1=0$.
\begin{lemma}\label{lem_ss}
	Assume that $\f\in\mathcal{P}(X_T)$ is a bounded pluripotential subsolution to \eqref{cmaf}. Then there exists a uniform constant $C>0$ such that for every $s\in(1-\varepsilon_0,1+\varepsilon_0)$,
	\begin{align*}
	(t,x)\mapsto u^s(t,x):=\frac{\alpha_s}{s}\f(ts,x)+(1-\alpha_s)\rho_1(x)-C|s-1|t
	\end{align*} is a pluripotential subsolution of~\eqref{eq cmaf} in $X_{T'}$.
\end{lemma}
\begin{proof}
	The proof is similar to that of \cite[Lemma 3.14]{guedj2020pluripotential}. Notice that we apply the mixed type inequality in the Hermitian setting; see Lemma~\ref{lem: mixed_ineq}.
\end{proof}

Let $\chi:\mathbb{R}\to[0,+\infty)$ be a smooth function with compact support in $[-1,1]$ such that $\int_\mathbb{R}\chi(s)ds=1$. For any $\varepsilon>0$, we set $\chi_\varepsilon(s):=\varepsilon^{-1}\chi(s/\varepsilon)$.
\begin{proposition}\label{prop: reg}
	Assume that $\f\in\mathcal{P}(X_T)$ is a bounded  pluripotential subsolution to \eqref{eq cmaf}. Let $u^s$ be defined as in Lemma \ref{lem_ss}. Then there exists a constant $B>0$ such that for all $\varepsilon>0$ small enough, the function
	\begin{align*}
	\Phi^\varepsilon(t,x):=\int_\mathbb{R}u^s(t,x)\chi_\varepsilon(s-1)ds-B\varepsilon(t+1)
	\end{align*} is a pluripotential subsolution of \eqref{eq cmaf} which is $\mathcal{C}^1$ in $t$ and such that
	\begin{align*}
	\sup_X[\Phi^\varepsilon(0,\cdot)-\f_0]\xrightarrow{\varepsilon\to 0}0.
	\end{align*}
\end{proposition}
\begin{proof} We use Lemma~\ref{lem_ss} and proceed exactly 
the same as in~\cite[Proposition 3.16]{guedj2020pluripotential}.
\end{proof}
	\subsection{Uniqueness}	\label{sect: unique_stable}
	We first establish the following comparison principle:
	\begin{proposition}\label{comparison1}
		Let $\f\in\mathcal{P}(X_T,\omega)\cap L^\infty(X_T)$ (resp. $\psi$) be a pluripotential subsolution (resp. supersolution) to \eqref{eq cmaf} with initial data $\f_0$ (resp. $\psi_0$). Assume that
		\begin{itemize}
			\item $\f_t\to\f_0$ and $\psi_t\to\psi_0$ in $L^1$ as $t\to 0$,
			
			\item $x\mapsto \f(\cdot,x)$ is continuous in $\Omega$ and $|\partial_t(t,x)|\leq C/t$ for all $(t,x)\in X_T$;
			
			\item $\psi$ is locally uniformly semi-concave in time $t\in (0,T)$;
			
		\end{itemize}
		Then \[\f_0\leq \psi_0\Rightarrow \f\leq \psi \,\, \text{on}\; X_T. \]
	\end{proposition}
	The proof below relies on the arguments in~\cite{guedj2020pluripotential} (see also~\cite{guedj2018stability}). 
	\begin{proof}
		Fix $0<T'<T$, we will prove that $\f\leq \psi$ on $[0,T']\times X$. The proof then follows by letting $T'\to T$. Using the invariance properties of the set of assumptions (see \cite[Sect. 3.3]{guedj2020pluripotential}) we may assume that the function $r\mapsto F(\cdot,\cdot,r)$ is increasing in the last variable.
We divide the proof in several steps.

\medskip

\noindent\textbf{Step 1.}\label{step 1} We assume moreover that
\begin{itemize}
    \item $\f$ is $\mathcal{C}^1$ in $t\in (0,T)$;
    \item the function $(t,x)\mapsto\psi(t,x)$ is continuous on $[0,T)\times \Omega$.
\end{itemize}
The proof, in this step, relies on the arguments in~\cite[Proposition 4.2]{guedj2020pluripotential}. The only difference is that we will use the minimum principle established in the previous section.
\smallskip

We recall that $\rho$ is a $\theta$-psh function with analytic singularities such that $\theta+\dc\rho \geq \delta\omega_X$ for some $\delta>0$, and $\Omega=\{\rho>-\infty\}$. In particular, $\rho$ is smooth in $\Omega$ and $\rho=-\infty$ on $\partial\Omega$.
The standard strategy is to work with $(1-\lambda)\f_t+\lambda\rho$ instead of $\f_t$. However, the time derivative $\dot{\f}_t$ may blow up at time zero, so we really need another auxiliary function. It was shown in~\cite[Theorem 3.4]{guedj2021quasi} that there exist $c=c(\theta,f)>0$ and $\phi\in\PSH(X,\theta)\cap L^\infty(X)$ normalized by $\sup_X\phi=0$ such that
\begin{equation*}
    (\theta+\dc\phi)^n=cfdV_X.
\end{equation*}
		Fix $\lambda,\varepsilon>0$ small enough and consider $$w(t,x):=(1-\lambda)\f(t,x)+\lambda\frac{\rho(x)+\phi(x)}{2}-\psi(t,x)-2\varepsilon t,\quad  (t,x)\in [0,T']\times  X.$$
		This function is upper semi-continuous on $[0,T']\times \Omega$, and goes to $-\infty$ on $\partial\Omega$.  Hence $w$ attains its maximum at some point $(t_0,x_0)\in [0,T']\times \Omega$. We claim that $w(t_0,x_0)\leq 0$. Suppose by contradiction that $w(t_0,x_0)>0$, in particular $t_0>0$. Set \[K:=\{x\in X: w(t_0,x)=w(t_0,x_0) \}. \]
		Since $\f(\cdot,x_0)$ is differentiable in $(0,T)$, the classical maximum principle insures that for all $x\in K$,
		\begin{align*}
		(1-\lambda)\partial_t\f(t_0,x)\geq \partial_t^-\psi(t_0,x)+2\varepsilon.
		\end{align*}
		By assumption the partial derivative $\partial_t\f (t_0,x)$ is continuous on $\Omega$. By local semi-concavity of $t\mapsto\psi_t$ it then follows that for every $t\in (0,T)$, $\partial_t^-\psi(t,\cdot)$ is upper semi-continuous in $\Omega$. Moreover, by Proposition~\ref{prop: continuous}, $\psi_t$ is continuous on $\Omega$  for each $t\in(0,T)$.
		We thus can find $\delta>0$ so small that, by introducing
		an open neighborhood  of $K$ 
		\begin{equation*}
		    D:=\{x\in\Omega: w(t_0,x)>w(t_0,x_0)-\delta \}\subset\subset \Omega,
		\end{equation*}
		 one has
		\begin{align}\label{41}
		(1-\lambda)\partial_t\f(t_0,x)> \partial_t^-\psi(t_0,x)+\varepsilon,\,\forall\, x\in D.
		\end{align} 
		Set $u:=(1-\lambda)\f(t_0,\cdot)+\lambda\frac{\rho+\phi}{2}$ and $v=\psi(t_0,\cdot)$. Since $\f$ is a subsolution and $\psi$ is a supersolution of~\eqref{eq cmaf} we infer
		\begin{align*}
		(\omega_{t_0}+\dc u)^n\geq e^{F(t_0,x,u(x))-F(t,x,v(x))+\varepsilon}(\omega_{t_0}+\dc v)^n
		\end{align*} in the weak sense of measures in $D$, where we have used \eqref{41}. Recall that $u(x)>v(x)+\varepsilon t_0$ for all $x\in K$. Shrinking $D$ if necessary, we can assume that  $u(x)>v(x)$ for all $x\in D$. Since $F$ is increasing in $r$ we thus infer
		\begin{align*}
		(\omega_{t_0}+\dc u)^n\geq e^\varepsilon(\omega_{t_0}+\dc v)^n
		\end{align*} in the weak sense of measures in $D$. 
		It follows from~the minimum principle (see Proposition~\ref{min-prin})  that \begin{align*}
		\min\{x\in\bar{D}:v(x)-u(x) \}=\min\{x\in \partial D: v(x)-u(x) \}.
		\end{align*}
		In particular, for all $x\in D$,
		\begin{align}\label{42}
		u(x)-v(x)+\min_{\partial D}(v-u)\leq 0.
		\end{align}
		Since $K\cap\partial D=\varnothing$, we infer $w(t_0,x)<w(t_0,x_0)$ for all $x\in \partial D$, hence
		\begin{align*}
		u(x)-v(x)<u(x_0)-v(x_0),\, \forall \, x\in \partial D
		\end{align*}
		which contradicts \eqref{42}. Altogether this shows that $t_0=0$, therefore $$(1-\lambda)\f(t,x)+\lambda\frac{\phi(x)+\rho(x)}{2}- \psi(t,x)-2\varepsilon t\leq \lambda\sup_X|\f_0|$$ in $[0,T']\times \Omega$. Letting $\lambda\to 0$ and then $\varepsilon\to 0$ we obtain  $\f\leq \psi$ in $[0,T']\times \Omega$, hence in $[0,T']\times X$.

	\smallskip

\noindent\textbf{Step 2.} We finally get rid of the assumptions in Step~\href{step 1}{1}. For the assumption that $\f$ is $\mathcal{C}^1$ in $t$, we use Proposition~\ref{prop: reg} and proceed the same way as in~\cite[Theorem 4.1]{guedj2020pluripotential}. 

To remove the assumption on $\p$, the argument is the same as that of \cite[Lemma 3.15, Proposition 4.4]{guedj2020pluripotential}.


Moreover, we can get rid of the assumption $\f_0\leq \psi_0$. Indeed, if $M_0:=\sup_X(\f_0-\psi_0)$ then $\f_0-M_0$ is a subsolution of the same equations since $F$ is increasing in the last variable. Hence $\f-M_0\leq \psi$ in $X_T$, i.e.,
		\begin{align*}
		\sup_{X_T}(\f-\psi)\leq \sup_X(\f_0-\psi_0)_+.
		\end{align*}
\end{proof}

The proof of Theorem~\ref{thmB} follows immediately from the comparison principle that we have just established.
\subsection{Stability}
We assume here that 
\begin{itemize}
    \item $F,G: [0,T)\times X\times \mathbb{R}\to\mathbb{R}$ are continuous;
    \item $F$ and $G$ are increasing in the last variable;
    \item $F$ and $G$ are uniformly Lipschitz in the last variable with Lipschitz constants $L_F$ and $L_G$,
    \item $0\leq f,g\in L^p(X)$ for some $p>1$, 
    \item the family $(\omega_t)_{t\in [0,T)}$ is assumed to be as in the introduction; see~\eqref{assum: omega}.
\end{itemize} 

The Lipschitz assumption on $F$ means that for all $(t,x)\in X_T$,
    \begin{equation*}
        |F(t,x,r)-F(t,x,r')|\leq L_F|r-r'|,\quad r,r'\in\mathbb{R}.
    \end{equation*}

\begin{proposition}\label{prop: stability}
  Let $\f\in \mathcal{P}(X_T,\omega)\cap L^\infty(X_T)$ (resp. $\psi$) be a subsolution (resp. supersolution) to~\eqref{cmaf} with data $(F,f,\omega_t,\f_0)$ (resp. $(G,g,\omega_t,\psi_0)$). Assume that $\psi$ is locally semi-concave in time $t\in (0,T)$ and is continuous on $(0,T)\times X$. Assume also that for each $t\in (0,T)$, $\f_t$ is continuous on $X$. 
   Fix $\varepsilon>0$. There exist constants $\alpha, A,B>0$ such that for all $(t,x)\in [\varepsilon,T)\times X$,
    \begin{equation*}
        \f(t,x)-\psi(t,x)\leq B\|\f_\varepsilon-\psi_\varepsilon\|_{L^1(X)}^\alpha+T\sup_{X_T\times \mathbb{R}}(G-F)_+ +A\|f-g\|^{1/n}_p,
    \end{equation*} where $A,B>0$ depend on $\theta$, $n$, $p$, a uniform bound $\f$, $\psi$, $\dot{\f}$, and $\dot{\psi}$ on $[\varepsilon,T)\times X$, $L_G$, and $\sup_{X_T}G(\cdot,\cdot,\sup_{X_T}|\f|)$.
\end{proposition}

\begin{proof}
  We are going to apply a perturbation argument as in \cite{guedj2018stability,guedj2020pluripotential} which was used in the work of Kolodziej~\cite{kolodziej1996sufficient}. If $\|f-g\|_p=0$ then $g\leq f$ almost everywhere on $X$. In this case, we can easily check that the function $(t,x)\mapsto \f(t,x)-Mt$ is a subsolution to~\eqref{cmaf} with data $(G,g,\omega_t,\f_0)$ and the proof thus follows from the comparison principle.
  
  Assume now $\|f-g\|_p>0$. It was shown in~\cite[Theorem 3.4]{guedj2021quasi} that there exist a constant $c_h>0$ and a function $\phi\in\PSH(X,\theta)\cap L^\infty(X)$, normalized by $\sup_X\phi=0$, such that
  \begin{equation*}
      (\theta+\dc \phi)^n=c_h hdV=c_h\left(1+\frac{|f-g|}{\|f-g\|_p}\right) dV.
  \end{equation*}
  Since $1\leq \|h\|_p\leq 2$, it follows from \cite[Lemma 3.3]{guedj2021quasi} that $c_h\geq c_1>0$ for some uniform constant $c_1$ (this means that $c_h$ is bounded from below away from 0).
  Set   
  \begin{equation*}
      M_0:=\sup_{X_T}\f,\quad M_1(\varepsilon):=\sup_{[\varepsilon,T)\times X}\dot{\f},\quad M_G=\sup_{X_T}G(t,x,M_0).
  \end{equation*}
  We also set $M_2(\varepsilon):=M_G+\max(L_GM_0,M_1(\varepsilon))$ and \begin{equation*}
     \delta:=\|f-g\|_p^{1/n}e^{(M_2(\varepsilon)-\ln c_1)/n}.
  \end{equation*}
  We may assume that $\delta$ is small enough. The computations in \cite[p. 1280-1281]{guedj2020pluripotential} shows that 
  \begin{equation*}
      \f_\delta(t,x):=(1-\delta)\f(t,x)+\delta\phi(x)+n\log(1- \delta)-B\delta t-Mt
  \end{equation*}
 is a subsolution to~\eqref{eq cmaf} on $[\varepsilon,T)$ with the data $(G,g)$, where  $B>0$ depends only on $n$, $\inf_{X_T}\f$ and $\inf_{[\varepsilon,T)\times X}\dot{\f}$, and $M>0$ is under control. 
 The comparison principle ensures that,  for all $(t,x)\in [\varepsilon,T)\times X$,
 \begin{equation*}
     \f(t,x)-\psi(t,x)\leq \max_X(\f(\varepsilon,\cdot)-\psi(\varepsilon,\cdot))_+ +TM+ A(\varepsilon)\|f-g\|_p^{1/n},
 \end{equation*}
 where $$A(\varepsilon):=(M_0+\|\phi\|_{L^\infty}+2n\log 2+BT)e^{(M_3(\varepsilon)-\ln c_1)/n}.$$ 
 Now we see that $\psi_\varepsilon$ is a supersolution to the degenerate elliptic equation
 \begin{equation*}
     (\theta+\dc\psi_\varepsilon)^n\leq e^{D(\varepsilon)}fdV
 \end{equation*} where $D(\varepsilon)$ is a upper bound of $\dot{\psi}(t,x)+G(t,x,\psi(t,x))$ on $[\varepsilon,T)\times X$. By Theorem~\ref{thm: stability}, there exists $\alpha>0$ depending on $n$, $p$ and a constant $C(\varepsilon)>0$ depending on $p$, $\theta$, $D(\varepsilon)$, and $\|f\|_p$ such that
 \begin{equation*}
     \max_X(\f_\varepsilon-\psi_\varepsilon)_+\leq C(\varepsilon)\left(\|(\f_\varepsilon-\psi_\varepsilon)\|_{L^1(X)}^\alpha+\|f-g\|_p^{1/n}\right).
 \end{equation*} This finishes the proof. 
\end{proof} 
We now finish the proof of Theorem~\ref{thmC}. Set $\Phi^j=\Phi(\omega_{t,j},F_j,f_j,\f_{0,j})$ and $\Phi=\Phi(\omega_t,F,f,\f)$. By Proposition~\ref{prop: stability} above, it suffices to show that the norm $\|\Phi^j_\varepsilon-\Phi_\varepsilon\|_{L^1(X)}$ goes to zero as $j\to+\infty$ for some small $\varepsilon>0$. Now we observe that the norm $\|\Phi^j_\varepsilon-\Phi_\varepsilon\|_{L^1(X)}$ is controlled by $\|\Phi^j-\Phi\|_{L^1([\varepsilon,T)\times X)}$ as follows from~\cite[Lemma 1.8]{guedj2018pluripotential}.   The proof  thus follows from Theorem~\ref{thm_211}. 

\section{Smoothness of the pluripotential flow}\label{sect: smooth}

In this section, we  establish, under some extra assumptions, a partial regularity of the pluripotential solution to the complex Monge-Amp\`ere flow~\eqref{cmaf} constructed in  previous sections. The assumptions on the density $f$ will become more transparent in the context of the Chern-Ricci flow on complex varieties with log terminal singularities. 
\subsection{Setup}
Recall that $(X,\omega_X)$ is a compact Hermitian manifold endowed with a reference hermitian form. Fixing $T\in (0,+\infty)$, let $(\omega_t)_{t\in [0,T]}$ be a smooth family of semi-positive forms such that for all $t\in [0,T]$, where $\theta\leq \omega_t$ with $\theta$  a semi-positive and volume non-collapsing $(1,1)$-form with big equivalence class, and
		\begin{equation}\label{assum: omega}
	    -A\omega_t\leq \dot{\omega}_t\leq A\omega_t\;\, \text{and}\;\, \ddot{\omega}_t\leq A\omega_t,\; \forall \, t\in [0,T].
	\end{equation} for some fixed constant $A>0$. Up to multiplying $\omega_X$ with a large constant, we may assume that $\omega_t\leq \omega_X$ for all $t\in[0,T]$.
We let $\Omega$ denote the "positive" locus of $\theta$; see Sect.~\ref{assumption}. 

In this section, we assume that $f$ is a positive measurable function on $X$ which is of the form
\begin{equation*}
    f=e^{\psi^+-\psi^-}
\end{equation*} where
\begin{itemize}
    \item $\psi^\pm$ are quasi-plurisubharmonic functions on $X$;
    \item $e^{-\psi^-}\in L^p$ for some $p>1$;
    \item $\psi^{\pm}$ are smooth in a given Zariski open subset $U\subset \Omega$.
\end{itemize}
Up to multiplying $\omega_X$ with a large positive constant, we may assume that  $\psi^{\pm}$ are both $\omega_X$-psh.
Assume furthermore that the function $F:[0,T]\times X\times\mathbb{R}\rightarrow\mathbb{R}$ is smooth. Given a $\omega_0$-psh function $\f_0$, our goal is to prove that there exists a unique bounded pluripotential solution $\f\in\mathcal{P}(X_T,\omega)\cap L^\infty(X_T)$ with $\f_t\to\f_0$ in $L^1(X)$ as $t\to 0^+$
and
such that on $(0,T)\times U$, $\f$ is smooth and satisfies 
\begin{equation*}
    (\omega_t+\dc\f_t)^n=e^{\partial_t\f_t+F(t,x,\f_t)}f(x)dV(x).
\end{equation*}

\subsection{A priori Laplacian estimates}
We are now going to estimate $\Delta\f_t$, where $\Delta$ denotes the Laplacian with respect to the reference metric $\omega_X$. The arguments follow the ones in \cite{boucksom2013regularizing} but there are extra difficulties coming from the torsion terms.

Since $U\subset \Omega$ is a Zariski open subset, we may choose a $\theta$-psh function $\rho_U$ with analytic singularities along $\partial U$ such that \begin{equation*}
    \theta+\dc\rho_U\geq 2\delta\omega_X
\end{equation*} for some $\delta>0$ (see e.g. \cite[Lemma 4.3.2]{boucksom2013regularizing}). In particular, $\rho_U\to -\infty$ near $\partial U$. We may normalize $\rho_U$ so that $\sup_X\rho_U= 0$.
\begin{theorem}\label{thm: laplacian}
  With the previous setup, assume also that $(\omega_t)_{t\in[0,T]}$ is a smooth path of \emph{Hermitian forms} on $X$. Assume moreover that $\psi^\pm$ are smooth $\omega_X$-psh functions on $X$ and 
  $
      \sup_X\psi^{\pm}\leq C
$ for some constant $C>0$. Let $\f_0$ be a smooth $\omega_0$-psh function on $X$. Suppose that $\f\in\mathcal{C}^\infty([0,T]\times X)$ satisfies the parabolic complex Monge-Amp\`ere equation
  \begin{equation*}
      \begin{cases}
        (\omega_t+\dc\f_t)^n=e^{\dot{\f}_t+F(t,\cdot,\f_t)+\psi^+-\psi^-}\omega_X^n\\
        \omega_t+\dc\f_t>0 \\
        \f|_{\{0\}\times X}=\f_0.
      \end{cases}
  \end{equation*}
  Then for each $\varepsilon_0>0$ and for each compact set $K\subset U$, there exists a constant $B>0$ only  depending on $K$, $\theta$, $\varepsilon_0$, $T$, $C$, $\sup_{[0,T]\times X}|\f|$, and $\sup_{[\varepsilon_0,T]\times X}|\dot{\f}|$ such that
  \begin{equation*}
    \sup_{[\varepsilon_0,T]\times K}  |\Delta\f_t|\leq Be^{-\delta\psi^-}.
  \end{equation*}
\end{theorem}
\begin{proof}
Fix $\varepsilon>0$ small.
By previous estimates, we have
\begin{equation*}
    \sup_{[0,T]\times X}|\f|\leq C_0-1,\qquad\sup_{[\varepsilon,T]\times X}|\dot{\f}|\leq C_0',
\end{equation*}
where $C_0$ explicitly depends on $A$, $\theta$, $\omega_X$, $T$, $\|f\|_{L^p}$, $\sup_X|\f_0|$, and $\sup_{X_T}F(t,x,0)$ while $C_0'$ depends on $C_0$, $\varepsilon$,  $\|\partial F/\partial r\|_{L^\infty}$.

We consider the function
\begin{equation*}
    H(t,\cdot):=t\log\tr_{\omega_X}(\tom_t)-\gamma(u_t)+\frac{1}{\f_{t+\varepsilon}+C_0},
\end{equation*}
where \begin{equation*}
    \tom_t=\omega_{t+\varepsilon}+\dc\f_{t+\varepsilon}, \qquad u_t(x)=\f_{t+\varepsilon}(x)-\rho_U(x)-\delta\psi^-(x)+C_0+\delta C\geq 1,
\end{equation*} and $\gamma:\mathbb{R}\to\mathbb{R}$ is a smooth concave increasing function such that $\lim_{t\to+\infty}\gamma(t)=+\infty$. We are going to show that $H$ is uniformly bounded from above for an appropriate choice of $\gamma$. The proof will thus follow since $u$ is uniformly bounded on compact subsets of $U$. 

We let $g$ denote the Hermitian metric associated to $\omega_X$ and $\Tilde{g}_t$ the one associated to $\tilde{\omega}_t=\omega_{t+\varepsilon}+\dc\f_{t+\varepsilon}$.
Since $\rho_U\to-\infty$ near $\partial U$, the function $H$  thus attains its maximum at some point $(t_0,x_0)\in [0,T]\times U$. If $t_0=0$ then $$\sup_{[0,T]\times U}H\leq -\gamma(1)+1.$$
It thus suffices to treat the case $t_0>0$. We may assume $\tr_{\omega_X}(\tom_t)\geq 1$ at $(t_0,x_0)$.

In order to apply the maximum principle, we need to perform a change of coordinates made possible by the following lemma from \cite[Lemma 2.1]{guan2010complex} (see also \cite[Lemma 2.9]{streets2011hermitian} for a similar argument). 
\begin{lemma}
    There exists a holomorphic coordinate system centered at $x_0$ such that for all $i$, $j$,
    \begin{equation*}
        g_{i\bar{j}}=\delta_{ij},\quad \frac{\partial g_{i\bar{i}}}{\partial z_j}=0
    \end{equation*} and also that the matrix $\Tilde{g}_{i\bar{j}}(t_0,x_0)$ is diagonal.
\end{lemma}
We let $\Delta_t=\tr_{\tom_t}\dc$ denote the Laplacian with respect to $\tom_t$.
Applying $\left(\frac{\partial}{\partial t}-\Delta_t\right)$ to $H$, 
\begin{equation}\label{eq: max}
    \begin{split}
        \left(\frac{\partial}{\partial t}-\Delta_t\right)H&= \log\tr_{\omega_X}(\tom_t)+t\frac{\tr_{\om_X}(\dot{\omega}_{t+\varepsilon}+\dc\dot{\f}_{t+\varepsilon})}{\tr_{\omega_X}(\tom_t)}-t\frac{\Delta_t\tr_{\omega_X}(\tom_t)}{\tr_{\omega_X}(\tom_t)}+t\frac{|\partial\tr_{\omega_X}(\tom_t)|_{\tom_t}^2}{(\tr_{\omega_X}(\tom_t))^2}\\
       &\quad-\gamma'(u)\dot{u}_t+\gamma'(u)\Delta_t u_t+\gamma''(u)|\partial u_t|^2_{\tom_t}\\
       &\quad-\frac{\dot{\f}_{t+\varepsilon}}{(\f_{t+\varepsilon}+C_0)^2}+\frac{\Delta_t\f_{t+\varepsilon}}{(\f_{t+\varepsilon}+C_0)^2}-\frac{2|\partial \f_{t+\varepsilon}|^2_{\tom_t}}{(\f_{t+\varepsilon}+C_0)^3}.
    \end{split}
\end{equation}
 Following computations in~\cite[Eq. (3.20)]{to2018regularizing} 
we obtain 
 \begin{equation}\label{eq: delta}
     \begin{split}
       \Delta_t\tr_{\omega_X}(\tom_t)&\geq \sum_{i,j,k}\Tilde{g}^{i\bar{i}}g^{j\bar{j}}\Tilde{g}^{j\bar{j}}\Tilde{g}_{i\bar{j}k}\Tilde{g}_{j\bar{i}\bar{k}}-\tr_{\omega_X}\textrm{Ric}(\tilde{\omega}_t)
          -C_1\tr_{\omega_X}(\tom_t)\tr_{\tom_t}(\omega_X).
     \end{split}
 \end{equation}
On the other hand, we have 
 \begin{equation*}\begin{split}
      \textrm{Ric}(\tom_t)&=\textrm{Ric}(\omega_X) -\dc(\dot{\f}_{t+\varepsilon}+F(t+\varepsilon,\cdot,\f_{t+\varepsilon}))-\dc(\psi^+-\psi^-)\\
      &\leq C_2\omega_X+\dc\psi^- -\dc(\dot{\f}_{t+\varepsilon}+F(t+\varepsilon,\cdot,\f_{t+\varepsilon}))
 \end{split}
 \end{equation*} with $C_2>0$ under control, using $\dc\psi^+\geq -C\omega_X$ and $\textrm{Ric}(\omega_X)\leq C\omega_X$.
 Plugging this into~\eqref{eq: delta} we have
 \begin{equation}
     \begin{split}
         \Delta_t\tr_{\omega_X}(\tom_t)&\geq \sum_{i,j}\Tilde{g}^{i\bar{i}}\Tilde{g}^{j\bar{j}}\Tilde{g}_{i\bar{j}j}\Tilde{g}_{j\bar{i}\bar{j}}-\tr_{\omega_X}(C_3\omega_X+\dc\psi^-)+\Delta(\dot{\f}_{t+\varepsilon}+F(t+\varepsilon,\cdot,\f_{t+\varepsilon}))\\ &\quad-C_3\tr_{\tom_t}(\omega_X) -C_4\tr_{\omega_X}(\tom_t)\tr_{\tom_t}(\omega_X),
     \end{split}
 \end{equation} which yields
 \begin{equation}\label{eq: delta2}
 \begin{split}
     \frac{\tr_{\om_X}(\dot{\omega}_t+\dc\dot{\f}_t)}{\tr_{\omega_X}(\tom_t)}-\frac{\Delta_t\tr_{\omega_X}(\tom_t)}{\tr_{\omega_X}(\tom_t)}&\leq\frac{C_5+\Delta\psi^- -\Delta F}{\tr_{\omega_X}(\tom_t)}-\frac{1}{\tr_{\omega_X}(\tom_t)} \sum_{i,j}\Tilde{g}^{i\bar{i}}\Tilde{g}^{j\bar{j}}\Tilde{g}_{i\bar{j}j}\Tilde{g}_{j\bar{i}\bar{j}}\\
     &\quad+C_3\frac{\tr_{\tom_t}(\omega_X)}{\tr_{\omega_X}(\tom_t)}+C_4\tr_{\tom_t}(\omega_X),
 \end{split}
 \end{equation} using $\dot{\omega}_t\leq A\omega_X$ and $-\dc\psi^+\leq C\omega_X$. Note here that $\Delta:=\textrm{tr}_{\omega_X}\dc$ denotes the Laplacian with respect to $\omega_X$.

From standard arguments as in~\cite[Eq. (4.5), p. 29]{guedj2021quasi},  we have  
 \begin{equation}\label{eq: error}
    \frac{{|\partial\tr_{\omega_X}(\tom_t)|_{\tom_t}^2}}{(\tr_{\omega_X}(\tom_t))^2}\leq \frac{1}{\tr_{\omega_X}(\tom_t)}\left(   \sum_{i,j}\Tilde{g}^{i\bar{i}}\Tilde{g}^{j\bar{j}}\Tilde{g}_{i\bar{j}j}\Tilde{g}_{j\bar{i}\bar{j}} \right)+C\frac{\tr_{\tom_t}(\omega_X)}{(\tr_{\omega_X}(\tom_t))^2}+\frac{2}{(\tr_{\omega_X}(\tom_t))^2}\textrm{Re}\sum_{i,j,k} \Tilde{g}^{i\bar{i}} T_{ij\bar{j}}\Tilde{g}_{k\bar{i}\bar{k}},
 \end{equation} where $T_{ij\bar{j}}:=\tilde{g}_{j\bar{j}i}-\tilde{g}_{i\Bar{j}j}$ is the torsion term corresponding to $\omega_{t+\varepsilon}$ which is under control: $|T_{ij\bar{j}}|\leq C$.
 Now at the point $(t_0,x_0)$, we have $\partial_{\Bar{i}}H=0$, hence
 \begin{equation*}
     t\sum_k \Tilde{g}_{k\bar{k}\Bar{i}}=\tr_{\omega_X}(\tom_t)\left(\gamma'(u)u_{\bar{i}}+\frac{\f_{\bar{i}}}{(\f_{t+\varepsilon}+C_0)^2}\right).
 \end{equation*}
 Cauchy-Schwarz inequality yields
 \begin{equation*}
     \begin{split}
         \left|\frac{2}{(\tr_{\omega_X}(\tom_t))^2}\textrm{Re}\sum_{i,j,k} \Tilde{g}^{i\bar{i}} T_{ij\bar{j}}\Tilde{g}_{k\bar{i}\bar{k}} \right|&=\frac{2}{(\tr_{\omega_X}(\tom_t))^2}\left|\textrm{Re}\sum_{i,j,k} \Tilde{g}^{i\bar{i}} T_{ij\bar{j}}\Tilde{g}_{k\bar{k}\bar{i}}+\textrm{Re}\sum_{i,j,k} \Tilde{g}^{i\bar{i}} T_{ij\bar{j}}T_{ik\bar{k}}\right|\\
         &\leq C\frac{\gamma'(u)^2}{-t\gamma''(u)}\frac{\tr_{\tom_t}(\omega_X)}{(\tr_{\omega_X}(\tom_t))^2}+(-\gamma''(u))\frac{|\partial u|^2_{\tom_t}}{t}\\
         &\quad + C\frac{\tr_{\tom_t}(\omega_X)}{t(\tr_{\omega_X}(\tom_t))^2} + \frac{|\partial\f_{t+\varepsilon}|^2_{\tom_t}}{t(\f_{t+\varepsilon}+C_0)^3}  +2C\frac{\tr_{\tom_t}(\omega_X)}{(\tr_{\omega_X}(\tom_t))^2}
     \end{split}
 \end{equation*} where we have used that $|T_{ij\Bar{j}}|\leq C$.  Together with~\eqref{eq: error} this yields at $(t_0,x_0)$,
 \begin{equation}\label{eq: error2}
     \begin{split}
         \frac{{|\partial\tr_{\omega_X}(\tom)|_{\tom_t}^2}}{(\tr_{\omega_X}(\tom_t))^2}&\leq \frac{1}{\tr_{\omega_X}(\tom_t)}\left(   \sum_{i,j}\Tilde{g}^{i\bar{i}}\Tilde{g}^{j\bar{j}}\Tilde{g}_{i\bar{j}j}\Tilde{g}_{j\bar{i}\bar{j}} \right)+\frac{|\partial\f_{t+\varepsilon}|^2_{\tom_t}}{t(\f_{t+\varepsilon}+C_0)^3}\\
         &\quad
          +C_6\left(\frac{\gamma'(u)^2}{-t\gamma''(u)}+\frac{1}{t}+2\right)\frac{\tr_{\tom_t}(\omega_X)}{(\tr_{\omega_X}(\tom_t))^2}+(-\gamma''(u))\frac{|\partial u_t|^2_{\tom_t}}{t}.
     \end{split}
 \end{equation}
 On the other hand, we observe that
 \begin{equation}\label{eq: lap_u}
     \gamma'(u)\Delta_tu=\gamma'(u)(n-\tr_{\tom_t}(\omega_{t+\varepsilon}+\dc(\rho_U+\delta\psi^-))\leq \gamma'(u)(n-\delta\tr_{\tom_t}(2\omega_X+\dc\psi^-)).
 \end{equation}
 Plugging  \eqref{eq: delta2}, \eqref{eq: error2} and~\eqref{eq: lap_u} into~\eqref{eq: max}, we obtain at $(t_0,x_0)$
 \begin{equation}\label{eq: max2}
     \begin{split}
         0\leq \left(\frac{\partial}{\partial t}-\Delta_t\right)H &\leq\log\tr_{\omega_X}(\tom_t) +t\frac{C_5+\Delta\psi^- -\Delta F}{\tr_{\omega_X}(\tom_t)}-\delta\gamma'(u)\tr_{\tom_t}(\omega_X+\dc \psi^-)
          \\
         &\quad +(C_7 T-\delta\gamma'(u))\tr_{\tom_t}(\omega_X) +C_6\left(\frac{\gamma'(u)^2}{-\gamma''(u)}+C+2T\right)\frac{\tr_{\tom_t}(\omega_X)}{(\tr_{\omega_X}(\tom_t))^2}\\
         &\quad+ n\gamma'(u)-\gamma'(u) \dot{\f}_{t+\varepsilon}-\frac{\dot{\f}_{t+\varepsilon}}{(\f_{t+\varepsilon}+C_0)^2}+\frac{\Delta_t\f_{t+\varepsilon}}{(\f_{t+\varepsilon}+C_0)^2}-\frac{|\partial \f_{t+\varepsilon}|^2_{\tom_t}}{(\f_{t+\varepsilon}+C_0)^3},
     \end{split}
 \end{equation} using that $\tr_{\omega_X}(\tom_t)\geq 1$ at $(t_0,x_0)$. 
 We now take care of the term $\Delta \psi^-$. Observe that
 
 $$0\leq \omega_X+\dc\psi^-\leq \tr_{\tom_t}(\omega_X+\dc\psi^-)\tom_t,$$
 and we take the trace with respect to $\omega_X$
 \begin{equation}\label{eq: lap_p}
     0\leq \frac{n+\Delta_{}\psi^-}{\tr_{\omega_X}(\tom_t)}\leq \tr_{\tom_t}(\omega_X+\dc \psi^-).
 \end{equation}
  For the Laplacian of $F$, we note that
\begin{equation*}
    \Delta F(\cdot,z,\f)=\Delta F(\cdot,z,\cdot)+2 Re\left(\sum g^{i\Bar{j}} \left(\frac{\partial F}{\partial r}\right)_i\f_{\bar{j}}\right)+\frac{\partial F}{\partial r}\Delta\f+\frac{\partial^2 F}{\partial r^2}|\partial\f|_{\omega_X}^2.
\end{equation*}
On the right hand-side, the first term is uniformly bounded on $X$, the last one is non-negative by the convexity assumption of $F$. Next, there is a constant $C>0$ depending on $\sup_{X_T}|\f|$ such that
\begin{align*}
    -\frac{\partial F}{\partial r}\Delta\f\leq C{\rm tr}_{\omega_X}(\tom_t)
\end{align*} 
since $\omega_t$ is semi-positive, and we have assumed that ${\rm tr}_{\omega_X}(\tom_t) \geq 1$.  Next, there exists a constant $C>0$ depending on $\sup_{X_T}|\f|$ such that
\begin{equation*}
    \begin{split}
        -2 Re\left(\sum g^{i\Bar{j}} \left(\frac{\partial F}{\partial r}\right)_i\f_{\bar{j}}\right)\leq |\partial \f|^2_{\tom_t}+C{\rm tr}_{\omega_X}(\tom_t),
    \end{split}
\end{equation*} which can be proved by using Arithmetic-Geometric Mean inequality.
Altogether, we thus obtain
\begin{equation*}
    -\frac{\Delta F(t,\cdot,\f)}{\text{tr}_{\omega}(\tom_t)}\leq C{\rm tr}_{\tom_t}(\omega_X)+C+\frac{|\partial\f|^2_{\tom_t}}{{\rm tr}_{\omega_X}(\tom_t)}
\end{equation*} using again $\text{tr}_{\omega_X}(\tom_t)\text{tr}_{\tilde{\omega}_t}(\omega_X)\geq n^2$.  Together with~\eqref{eq: lap_p}, we obtain
\begin{equation*}
    \frac{C_5+\Delta\psi^- -\Delta F}{\tr_{\omega_X}(\tom_t)}\leq \tr_{\tom_t}(\omega_X+\dc \psi^-) +C_8{\rm tr}_{\tom_t}(\omega_X)+\frac{|\partial\f|^2_{\tom_t}}{{\rm tr}_\omega(\tom_t)}.
\end{equation*}
From this, we have
\begin{equation}\label{eq: lap_F}\begin{split}
   t\frac{C_5+\Delta\psi^- -\Delta F}{\tr_{\omega_X}(\tom_t)}-\delta\gamma'(u)\tr_{\tom_t}(\omega_X+\dc \psi^-)&\leq(T-\delta\gamma'(u))\tr_{\tom_t}(\omega_X+\dc \psi^-) \\
   &\quad+C_8T{\rm tr}_{\tom_t}(\omega_X)+T\frac{|\partial\f|^2_{\tom_t}}{{\rm tr}_{\omega_X}(\tom_t)}.
\end{split}
\end{equation}
We now choose the function $\gamma$ as follows $$\gamma(s):=Bs-\frac{1}{s}$$ with $B>0$ so large that $B\geq 3+(C_7+1)T\delta^{-1}$.  Since $u\geq 1$, we have
\begin{equation*}
    B\leq \gamma'(u)\leq 1+B, \qquad C_6\left(\frac{\gamma'(u)^2}{-\gamma''(u)}+C+2T\right)\leq C_{9}u^3.
\end{equation*}
Plugging this together with~\eqref{eq: lap_F} into \eqref{eq: max2} we obtain
\begin{equation}\label{eq: max3}
    \begin{split}
        0&\leq \log\tr_{\omega_X}(\tom_t)-3\tr_{\tom_t}(\omega_X)+C_9u^3\frac{\tr_{\tom_t}(\omega_X)}{(\tr_{\omega_X}(\tom_t))^2}\\
       & \quad +\left(\frac{T}{{\rm tr}_{\omega_X}(\tom_t)}-\frac{1}{(\f_{t+\varepsilon}+C_0)^3}\right)|\partial\f_{t+\varepsilon}|^2_{\tom_t}+C_{10}
    \end{split}
\end{equation} with $C_{10}$ only depending on $n$, $\sup_{X_T}|\f_t|$ and $\sup_{X_T}|\dot{\f}_{t+\varepsilon}|$, where have used $\Delta_t\f_{t+\varepsilon}\leq n$. 
If  we have at $(t_0,x_0)$, $\tr_{\omega_X}(\tom_t)\leq \sqrt{C_9 u_t^3}+T(\f_{t+\varepsilon}+C_0)^3$
 then at the same point
 \begin{equation*}
     H\leq T\log (\sqrt{C_9 u^3}+ 8TC_0^3)-\gamma(u)+1\leq C_{11},
 \end{equation*} with $C_{11}>0$ under control, and we are done. Otherwise at $(t_0,x_0)$ we have $\tr_{\omega_X}(\tom_t)\geq \sqrt{C_9 u^3}+T(\f_{t+\varepsilon}+C_0)^3$, and then \eqref{eq: max3} becomes
 \begin{equation}\label{eq: max4}
     0\leq \log\tr_{\omega_X}(\tom_t)-2\tr_{\tom_t}(\omega_X)+C_{10}.
 \end{equation}
Now it follows from \cite[Lemma 4.1.1]{boucksom2013regularizing} that 
\begin{equation}\label{ineq: bg13}
    \log\tr_{\omega_X}(\tom_t)\leq (n-1)\log\tr_{\tom_t}(\omega_X)+C_{12}-\psi^{-}
\end{equation} with $C_{12}$ depending on $\sup|\f|$, $\sup|\dot{\f}|$ and $\sup_X\psi^+$. Together with~\eqref{eq: max4} this yields at $(t_0,x_0)$
\begin{equation}
    \label{eq: max5}
    0\leq -\tr_{\tom_t}(\omega_X)-\psi^{-}+C_{13}
\end{equation} since $(n-1)\log x-2x\leq -x+O(1)$ for $x>0$. Plugging this into~\eqref{ineq: bg13}, it follows that 
\begin{equation*}
    \log\tr_{\omega_X}(\tom_t)\leq (n-1)\log(C_{13}-\psi^-)+C_{12}-\psi^{-}
\end{equation*} at $(t_0,x_0)$. Therefore, 
\begin{equation*}\begin{split}
    H(t_0,x_0)&\leq t_0(n-1)\log(C_{13}-\psi^-)+C_{12}+(B\delta-1)\psi^{-}-B(\f-\rho_U)+1+\frac{1}{\f+C_0}.
\end{split}
\end{equation*} Now, using the obvious inequality $-bx+\log x\leq O(1)$ for $x\in (0,+\infty)$ we get a uniformly upper bound for $H$ at the point $(t_0,x_0)$. Applying this for $\varepsilon=\varepsilon_0/2$ to finish the proof.

\end{proof}

\subsection{Higher order estimates}
Thanks to Demailly's regularisation theorem \cite{demailly1992regularization}, there exists two sequences $\psi^{\pm}_j\in\mathcal{C}^\infty(X)$ such that \begin{itemize}
    \item $\psi^{\pm}_j$ decreases pointwise to $\psi^\pm$ on $X$, the convergence is in $\mathcal{C}^\infty$ topology on $U$, $\sup_X\psi^\pm_j\leq C$;
    \item $\dc\psi^\pm_j\geq -B\omega$ for some $B>0$ under control.
\end{itemize} In fact, the constant $B>0$ depends on the Lelong numbers of quasi-psh functions $\psi^\pm$ according to  Demailly's result. Nevertheless, these Lelong numbers can be uniformly bounded in terms of the lower bound $-C\omega$ for $\dc\psi^\pm$ by standard arguments; see e.g.~\cite[Lemma 2.5]{boucksom2002volume}. Moreover, $\sup_X\psi^\pm_j$ is bounded independently of $j$, while we have for all $j\in\mathbb{N}$,
\begin{equation*}
    \|e^{-\psi^-_j}\|_{L^p}\leq  \|e^{\psi^-}\|_{L^p}.
\end{equation*}

Using Demailly's regularization theorem again, we also get a decreasing sequence of smooth functions $\f_0^j$ such that  $\f_0^j$ decreases pointwise to $\f_0$ as $j\to +\infty$, and $\omega_0+\dc \f_0^j\geq -\varepsilon_j\omega_X$ with $\varepsilon_j\to 0$. We then set $$\omega_t^j:=\omega_t+\varepsilon_j\omega_X,\quad f_j=e^{\psi^+_j-\psi^-_j}.$$ 
Since for each $j$, $(\omega_t^j)$ is a smooth family of Hermitian forms, and  $f_j$ is smooth. Moreover, $\f_0^j$ is smooth and strictly $\omega_0^j$-psh. It follows from~\cite{tosatti2015evolution,to2018regularizing} that there exists a unique function $\f^{j}\in \mathcal{C}^\infty([0,T]\times X)$ such that \begin{equation}\label{eq: approx_cmaf}
    \begin{cases}
    (\omega_t^j+\dc\f^{j}_t)^n=e^{\dot{\f}^{j}+F(t,x,\f^{j})}f_j dV\\
    \f^{j}|_{\{0\}\times X}=\f_0^j.
    \end{cases}
\end{equation}

As in the proof of Theorem~\ref{thm: existence},  up to extracting and relabelling, there exists a (unique) bounded function $\f$ such that $\f^j\to \f$ in $L^1_{\rm loc}(X_T)$ and 
\begin{equation*}
    (\omega_t+\dc \f_t)^n\wedge dt=e^{\dot{\f}_t+F(t,x,\f)}f(x)dV(X)\wedge dt
\end{equation*} in the sense of measures on $X_T$. The convergence $\f^j\to\f$ is moreover  locally uniform in $X_T$.
We are going to prove that $\f_t$ is smooth in $U$ for each $t>0$. 

Observe that $f_j=e^{\psi^+_j-\psi^j_j}\leq e^{C-\psi^-_j}$ is uniformly bounded in $L^p$ and the family of Hermitian form $(\omega_t^j)$ satisfies
\begin{equation*}
 \theta\leq \omega_t^j\leq A\omega_X,\qquad   -A\omega_t^j\leq \dot{\omega}_t^j\leq A\omega^t_j
\end{equation*} for some uniform constant $A>0$.
Fix $\varepsilon>0$.
From the previous section, we have a  uniform bounds for $\f_t^j$ and $\dot{\f}_t^{j}$ on $[\varepsilon,T]\times X$:
\begin{equation*}
    \sup_{[\varepsilon,T]\times X}\left( |{\f}_t^{j}|+\left|\frac{\partial{\f}_t^{j}}{\partial t} \right| \right)\leq C_{\varepsilon,T},
\end{equation*} for some constant $C_{\varepsilon,T}$.
Fix a compact set $K\Subset U$. We apply Theorem~\ref{thm: laplacian} to have a uniform Laplacian estimates for $\f_t^{j}$ on $[\varepsilon,T]\times K$:
\begin{equation*}
    \sup_{[\varepsilon,T]\times K}|\Delta\f_t^{j}|\leq C_{\varepsilon,T,K}.
\end{equation*}
 Then using the complex parabolic Evans-Krylov theory and Schauder's estimates, we obtain higher order estimates for $\f_t^j$ on $[\varepsilon,T]\times K$. Namely, for each $m\in\mathbb{N}$ there exists $C_{m,\varepsilon,K}>0$ such that
 \begin{equation*}
     \|\f^{j}\|_{\mathcal{C}^m([\varepsilon,T]\times K)}\leq C_{m,\varepsilon,K}.
 \end{equation*}
Up to extracting a subsequence, we have $\f^j\to \f$ in $L^1(X_T)$ such that on $(0,T)\times U$, $\f$ is smooth and satisfies
$$(\omega_t+\dc\f_t)^n=e^{\dot{\f}_t+F(t,x,\f_t)}f(x) dV(x)$$ 
with $\f_0=\lim_j\f_0^j$.
\section{Chern-Ricci flows on varieties with log terminal singularities}\label{sect: crf_lt}
In this section we prove the existence and uniqueness of the weak Chern-Ricci flow on complex varieties with log terminal singularities. 
We prove that the hypotheses from previous sections are satisfied, allowing one run the Chern-Ricci flow  from an arbitrary closed positive current with bounded potentials. 
\subsection{Notations}\label{sect: notionCRF}

We refer the reader to~\cite[Sect. 5]{eyssidieux2009singular} for basic facts on pluripotential theory on compact complex variety with log terminal singularities. Roughly speaking one can consider local embedding $j_\alpha:U_\alpha\hookrightarrow \mathbb{C}^N$ and objects 
that are restrictions of global ones.

We assume here that $Y$ is a $\mathbb{Q}$-Gorenstein variety, i.e. $Y$ is a normal complex space such that
its canonical divisor $K_Y$ is $\mathbb{Q}$-Cartier.
We denote the singular set of $Y$ by $Y_{\rm sing}$ and let $Y_{\rm reg}:=Y\setminus Y_{\rm sing}$.
Given a log resolution of singularities $\pi:X\to Y$ (which may and will always be chosen to be an isomorphism over $Y_{\rm reg}$ ), there exists  a unique (exceptional) $\mathbb{Q}$-divisor $\sum a_iE_i$ with simple normal  crossings (snc for short) such that
\begin{equation*}
    K_X=\pi^*K_Y+\sum_ia_iE_i,
\end{equation*} The coefficients $a_i\in\mathbb{Q}$ are called {\em discrepancies} of $Y$ along $E_i$. 
\begin{definition}
  We say that $X$ has \emph{log terminal} (\emph{lt} for short) singularities if and only if $a_i>-1$ for all $i$.  
\end{definition}
    
This definition is independent   of the choice of a log resolution. We refer the reader to \cite{kollar1998birational} for more details.
 
Let $m$ be a positive integer such that $mK_Y$ is Cartier. If we choose $\sigma$ a local generator of $mK_Y$ defined on an open subset $U$ of $Y$, then \begin{equation}\label{eq: sect4_1}
 (i^{mn^2}\sigma\wedge\Bar{\sigma})^{1/m}   
\end{equation}
defines a smooth positive volume form on $U\cap Y_{\rm reg}$. If $
f_i$ is a local equation of $E_i$ around a point $\pi^{-1}(U)$, then we can see that 
\begin{equation}\label{eq: adt}
  \pi^*\left(i^{mn^2}\sigma\wedge\Bar{\sigma} \right)^{1/m}=\prod_i|f_i|^{2a_i}dV  
\end{equation} locally on $\pi^{-1}(U)$ for some local volume form $dV$. Since $\sum E_i$ are simple normal crossing, this implies that $Y$ has lt singularities if and only if  each volume of the form \eqref{eq: sect4_1} has locally finite mass near singular points of $Y$.

The previous construction leads to the following definition of \emph{adapted measure} which is introduced in \cite[Sect. 6]{eyssidieux2009singular}:
\begin{definition}
    Let $h$ be a smooth hermitian metric on the $\mathbb{Q}$-line bundle $\mathcal{O}_Y (K_Y)$. The corresponding adapted measure $\mu_{Y,h}$ on $Y_{\rm reg}$ is locally defined by choosing a nowhere vanishing  section $\sigma$ of $mK_Y$ over a small open set $U$ and setting
    \begin{equation*}
        \mu_{Y,h}:=\frac{(i^{mn^2}\sigma\wedge\Bar{\sigma})^{1/m}}{|\sigma|_{h^m}^{2/m}}.
    \end{equation*}
\end{definition}  

The point of the definition is that the measure $\mu_{Y,h}$ does not depend on the choice of $\sigma$, so is globally defined. The arguments above show that $Y$ has lt singularities if and only if  $\mu_{Y,h}$ has finite total mass on $Y$, in which case we can consider it as a Radon measure on the whole of $Y$. Then $\chi=\dc\log\mu_{Y,h}$ is well-defined smooth closed $(1,1)$-form on $Y$. 


 \begin{remark}
 In \cite{song2017kahler}, Song and Tian defined an adapted measure of the form $\mu_{Y,h}$ for a smooth metric $h$ on $K_Y$ as a smooth volume form. We would like to avoid this terminology,  which has the drawback that $\omega^n$ might not be smooth in this sense even when $\omega$ is smooth positive $(1,1)$-form on $X$. This is in fact already the case for quotient singularities.
 \end{remark}
 Given a Hermitian form $\omega_Y$ on $Y$, there exists a unique hermitian metric $h=h(\omega_Y)$ of $K_Y$ such that 
 \[\omega_Y^n=\mu_{Y,h}.\]
We have the following definition.
\begin{definition}
    The \emph{Ricci curvature form} of $\omega_Y$ is $\textrm{Ric}(\omega_Y):=-\dc\log h$.
\end{definition}

We now define the notion of the weak Chern-Ricci flow on compact complex varieties with log terminal singularities. 
\begin{definition}
    Let $Y$ be a $\mathbb{Q}$-Gorenstein compact complex variety with log terminal singularities and let $\theta_0$ be Hermitian metric such that
    \begin{equation*}
        T_{\rm max}:=\sup \{t>0: \exists \psi\in\mathcal{C}^\infty(Y)\,\text{with}\; \theta_0-t\textrm{Ric}(\theta_0)+\dc\psi>0  \}>0.
    \end{equation*} Fix $S_0=\theta_0+\dc\phi_0$ a positive (1,1)-current with $\phi_0$ a bounded $\theta_0$-psh function.   
    A family $(\theta_t)_{t\in[0,T_{\rm max})}$ of semi-positive (1,1)-currents on $Y$ is said to be a solution of the \emph{weak Chern-Ricci flow} starting with $S_0$ if the following conditions hold:
    \begin{enumerate}
    \item $\{\theta_t\}= \{\theta_0\}-t c_1^{BC}(Y)\in H^{1,1}_{\rm BC}(X,\mathbb{R})$,
        \item $\theta_t\to S_0$ as $t\to 0$,
        \item $(\theta_t)_{t\in[0,T_{\rm max})}$ restrict to a smooth path of Hermitian forms on $ Y_{\rm reg}$ satisfying
        \begin{equation*}
            \frac{\partial\theta_t}{\partial t}=-\textrm{Ric}(\theta_t).
        \end{equation*} 
    \end{enumerate}
\end{definition}
\subsection{Existence and uniqueness of the weak Chern-Ricci flow}
Our aim is to establish the existence and uniqueness for the weak Chern-Ricci flow on complex variety with log terminal singularities. Furthermore, we obtain smoothing properties for the weak Chern-Ricci flow even if the initial data is not smooth. 
\begin{theorem}
  Let $Y$ be a $\mathbb{Q}$-Gorenstein compact complex variety with log terminal singularities and let $\theta_0$ be Hermitian metric
 such that  \begin{equation*}
        T_{\rm max}:=\sup \{t>0: \exists \psi\in\mathcal{C}^\infty(Y)\;\text{with}\; \theta_0-t\textrm{Ric}(\theta_0)+\dc\psi>0  \}>0.
    \end{equation*}   If $S_0=\theta_0+\dc\phi_0$ for some bounded $\theta_0$-psh function $\phi_0$, then there exists a unique solution $(\theta_t)_{t\in [0,T_{\rm max})}$ of the weak Chern-Ricci flow starting with $S_0$.
\end{theorem}

\begin{proof}
It is  classical  that solving  the (weak)  Chern-Ricci flow  is  equivalent to solving a complex Monge-Amp\`ere equation flow. 
Let $\chi$ be a closed smooth (1,1) form that represents $c^{\rm BC}_1(K_Y)$. Given $T\in (0,T_{\rm max})$, we set
\begin{equation*}
    \hat{\theta}_t:=\theta_0+t\chi
\end{equation*} which defines a affine path of semipositive (1,1)-forms with local potentials. Since $\chi$ is a smooth representative of $c_1^{\rm BC}(K_Y)$, one can find a smooth metric $h$ on the $\mathbb{Q}$-line bundle $\mathcal{O}_Y(K_Y)$ with curvature form $\chi$. We obtain $\mu_{Y,h}$ the adapted measure corresponding to $h$.
The Chern-Ricci flow is equivalent to the following Monge-Amp\`ere flow 
\begin{equation}\label{eq: maf}
    (\hat{\theta}_t+\dc\phi_t)^n=e^{\partial_t{\phi}}\mu_{Y,h}
\end{equation} where $\mu_{Y,h}$
 is the adapted measure on $Y$ corresponding to $h$. Since $\hat{\theta}_t$ is a Hermitian form, there is a small constant $c>0$ such that $\hat{\theta}_t\geq c\theta_0$ for all $t\in [0,T]$. 
 
Now let $\pi:X\to Y$ be a log resolution of singularities as defined in the previous subsection.  We have seen that the measure 
$$\mu:=fdV \quad\text{where}\;\; f=\prod_{i}|s_i|^{2a_i}$$ has  poles (corresponding to $a_i<0$) or zeroes (corresponding to $a_i>0$) along the exceptional divisors $E_i=(s_i=0)$, $dV$ is a smooth volume form. Passing to the resolution, the flow~\eqref{eq: maf} becomes
\begin{equation}\label{eq: cmaf}
    (\omega_t+\dc\f_t)^n=e^{\partial_t\f}\mu
\end{equation} where $\omega_t:=\pi^*\hat{\theta}_t$ and $\f:=\pi^*\phi$. Since $(\hat{\theta}_t)_{t\in [0,T]}$ is a smooth family of Hermitian forms, it follows that the family of semipositive forms $[0,T]\ni t\mapsto\omega_t$ satisfies all our requirements. We also have $\omega_t\geq \Tilde{\theta}:=\pi^*(c\theta_0)$, the latter is  smooth semipositive  and  big (see e.g.~\cite[Proposition 3.2]{fino2009blow}), but no longer hermitian. We fix a $\title{\theta}$-psh function $\Tilde{\rho}$ with analytic singularities along a divisor $E=\pi^{-1}(Y_{\rm sing})$ such that $\Tilde{\theta}+\dc\Tilde{\rho}\geq 2\delta\omega_X$ for some $\delta>0$. 

If we set $\psi^{+}=\sum_{a_i>0}2a_i\log|s_i|$, $\psi^{-}=\sum_{a_i<0}-2a_i\log|s_i|$, we observe that $\psi^{\pm}$ are quasi-psh functions with logarithmic poles along the exceptional divisors, smooth on $X\setminus \textrm{Exc}(\pi)=\pi^{-1}(Y_{\rm reg})$, and $e^{\psi^-}\in L^p(dV)$ for some $p>1$. Finally, we have
\begin{equation*}
   \{\Tilde{\rho}>-\infty\}=X\setminus \textrm{Exc}(\pi)=\pi^{-1}(Y_{\rm reg})\simeq Y_{\rm reg}.
\end{equation*} 
Applying Theorem~\ref{thmA} and~\ref{thmB}, there exists a unique pluripotential solution to the Monge-Amp\`ere flow on $X_T$ starting with $\f_0=\pi^*\phi_0$ for any fixed $T\in (0,T_{\rm max})$, denoted by $\f_t$.  The higher regularity properties of $\f$ follows from Theorem~\ref{thmD}.  Pushing this solution down to $Y$, we obtain a weak solution to the Chern-Ricci flow starting with $S_0$, denoted by $\phi_t$.

Now assume that $\psi_t$ is another solution to the weak Chern-Ricci flow on $Y$. Then $\pi^*\psi$ is a weak solution to the flow~\eqref{eq: cmaf} on $\pi^{-1}(Y_{\rm reg})$, it can be extended trivially on the whole $X$. The uniqueness result thus yields that $\f_t=\pi^*\psi_t$ on $X$, hence $\phi_t=\psi_t$ on $Y$.
\end{proof}

\bibliographystyle{alpha}
	\bibliography{bibfile}	
\end{document}